\documentclass[a4paper]{amsart}
  
\usepackage{graphicx}
\usepackage{cite}
\usepackage{hyperref}
\usepackage{soul}
\usepackage{psfrag}
\usepackage{stmaryrd}
\usepackage{cancel}
\usepackage{multirow}
\usepackage{xcolor}
\usepackage{amsmath} 
\usepackage{amssymb}

\DeclareMathAlphabet{\mathpzc}{OT1}{pzc}{m}{it}

\usepackage[percent]{overpic}

\graphicspath{{./figure/}{./}}

\newcommand{\Dx}{\Delta x}

\newcommand{\Ccal}{\mathcal{C}}
\newcommand{\Ical}{\mathcal{I}}

\newcommand{\Ncal}{\mathcal{N}}
\newcommand{\Ocal}{\mathcal{O}}
\newcommand{\Pcal}{\mathcal{P}}
\newcommand{\Rcal}{\mathcal{R}}
\newcommand{\Scal}{\mathcal{S}}

\newcommand{\phit}{\tilde\phi}
\newcommand{\Ct}{\tilde \Ccal}

\newcommand{\sh}{h^{b}}
\newcommand{\su}{u^{b}}
\newcommand{\sv}{v^{b}}
\newcommand{\sq}{q^{b}}

\newcommand{\Fr}{\mathcal{F}}

  \newtheorem{df}{Definition}[section]
  
  \newtheorem{prop}[df]{Proposition}

  \theoremstyle{definition}
  \newtheorem{remark}[df]{Remark}

\theoremstyle{definition}
  
 \theoremstyle{definition}
  
\begin{document}

\title[Open canals flow with fluvial to torrential phase transitions on networks]{
Open canals flow with fluvial to torrential phase transitions on networks
}
\author{{Maya Briani}}\thanks{Istituto per le Applicazioni del Calcolo ``M. Picone'', Consiglio Nazionale delle Ricerche, Via dei Taurini 19, 00185 Rome, Italy (m.briani@iac.cnr.it)}
\author{{Benedetto Piccoli}}\thanks{Department of Mathematical Sciences, Rutgers University--Camden 311 N. 5th Street Camden, NJ 08102 (piccoli@camden.rutgers.edu)}

\date{}

\maketitle

\begin{abstract} 
Network flows and specifically open canal flows can be modeled by
systems of balance laws defined on topological graphs.
The shallow water or Saint-Venant system of balance laws is one of the most used model
and present two phases: fluvial or sub-critical and torrential or super-critical.
Phase transitions may occur within the same canal but transitions related
to networks are less investigated.
In this paper we provide a complete characterization of possible phase transitions
for a simple network with two canals and one junction.
Our analysis allows the study of more complicate scenarios.
Moreover, we provide some numerical simulations to show the theory at work.
\end{abstract}

\section{Introduction}
The dynamics of network flows is usually modelled by systems of
Partial Differential Equations (briefly PDEs), most of time balance laws.
The dynamics is defined on a topological graph with evolution on arcs given
by system of PDEs, while additional conditions must be assigned at network nodes,
e.g. conservation of mass and momentum. 
There is a large literature devoted to these problems and we refer to \cite{BCGHP14}
for a extensive survey and for additional references.

In particular, here we focus on water flows on a oriented network of open canals
and the model given by Saint-Venant or shallow water equations.
The latter form a non linear system of balance laws
composed by a mass and momentum balance laws. 
In water management problems, these equations are often used as a fundamental tool to describe the dynamics of canals and rivers, 
see \cite{BBDLP09} and papers in same volume,
and various control techniques were proposed, see \cite{BC16,BCDA09,GL09,GL13,Hante17,LITRICO2005,PW18} and references therein.
Moreover, the need of dynamic models in water management
is well documented, see \cite{Milly573}.
The shallow water system is hyperbolic (except when water mass vanishes)
and has two genuinely nonlinear characteristic fields.
Moreover, it exhibits two regimes: \emph{fluvial} or sub-critical, when
one eigenvalue is negative and one positive, and \emph{torrential} or super-critical, 
when both eigenvalues are positive.
This is captured by the so called \emph{Froude number}, see (\ref{eq:FN}).
For a complete description of the physics of the problem one needs to supply the equations with conditions at nodes, which represent junctions.
The junction conditions are originally derived by engineers in the modeling of the dynamic of canals and rivers. 
The first and most natural condition is the
conservation of water mass which is expressed as the equality between the sum of fluxes from the incoming canals
and that from outgoing ones. 
One single condition is not sufficient to isolate a unique solution,
thus different additional condition were proposed in the literature.
Physical reasons motivate different choices of conditions, among which the equality of water levels, of energy levels and conservation of energy.
For the assessment of coupling conditions on canals networks and for more details on the existence of solutions in the case of subcritical flows, we refer the reader to \cite{CHS2008,GLS2004,G2005,HS2013,Kreiss2013,LS2002,Marigo2010}. For discussion on supercritical flow regimes, we refer the reader to \cite{GLS2004} and references there in.\\
Then, to construct solutions one may resort to the concept
of Riemann solver at a junction, see \cite{GP09}. A Riemann solver at a junction
is a map assigning solutions to initial data which are constant on each arc.
Alternatively one may assign boundary conditions on each arc, but,
due to the nonlinearity of equations, one has to make sure that boundary
values are attained. This amounts to look for solutions with waves having
negative speed on incoming channels and positive on outgoing ones:
in other words waves do not enter the junction.
A Riemann solver with such characteristics is called consistent, see also \cite{GP06}.

In this paper we are interested in transitions between different flow regimes, when the transition occurs at a junction of a canals network. We assume to have \textit{incoming canals} which end at the junction and \textit{outgoing canals} which start at the junction. 
Thus we formulate a left-half Riemann problem for incoming canals and a right-half Riemann problem for outgoing canals to define the region of admissible states such that waves do not propagate into the junction.
This corresponds to identify the regions where Riemann solvers can take values
in order to be consistent. 
Such regions are enclosed by the Lax curves (and inverted Lax curves)
and the regime change curves. To help the geometric intuition, we developed
pictures showing such curves and the regions they enclose.\\
Then we consider the specific case of two identical canals interconnected at a junction (\textit{simple junction}). We start focusing on conservation
of water through the junction and equal height as coupling conditions.
It is typically expected the downstream flow to be more regular,
thus we consider three cases:  fluvial to fluvial,  torrential to fluvial
and torrential to torrential.
In the fluvial to fluvial case there exists a unique
solution. However such solution may be different than the solution
to the same Riemann problem inside a canal (without the junction) and may exhibit the appearance of a torrential regime.
The torrential to fluvial case is more delicate to examine. Three different
cases may happen: the solution propagates the fluvial regime upstream,
the solution propagates the torrential regime downstream
or no solution exits.
Finally, in the torrential to torrential case, if the solution exists then
it is torrential.\\
To illustrate the achieved results we perform simulations using a Runge-Kutta Discontinuous Galerkin scheme \cite{BPQ2016}. 
The RKDG method is an efficient, effective and compact numerical approach for simulations of water flow in open canals.
Specifically, it is a high-order scheme and compact in the sense that the solution on one computational cell depends only on direct neighboring cells via numerical fluxes, thus allowing for easy handling the numerical boundary condition at junctions. 
In the first example we show a simulation where an upstream
torrential regime is formed starting from special fluvial to fluvial conditions.
The second example shows how a torrential regime may propagate downstream.\\
We conclude by discussing the possible solutions if the water height condition
is replaced by the equal energy condition.



The paper is organized as follows: in Section \ref{sec:model}, we present the model starting from the one-dimensional shallow water equations. In Section \ref{sec:notations}, we give useful notations and preliminary results that allow to  determine the admissible states for the half-Riemann problems discussed in the following Section \ref{sec:half-Riemann}. In Section \ref{sec:junction} we study possible solutions at a simple junction for different flow regimes and different junction conditions. Finally, in Section \ref{sec:num_tests} we illustrate the results of the previous section with a couple of numerical tests.

\section{Flow classification and governing equations}\label{sec:model}
The most common and interesting method of classifying open-channel flows is by dimensionless \textit{Froude number}, which for a rectangular or very wide channel is given by the formula:
\begin{equation}\label{eq:FN}
Fr = \frac{|v|}{\sqrt{gh}},
\end{equation}
where $v$ is the average velocity and $h$ is the water depth. The three flow regimes are:
\begin{itemize}
\item $Fr < 1$ \textit{subcritica}l flow or fluvial regime;
\item $Fr=1$ \textit{critical} flow;
\item $Fr>1$ \textit{supercritical} flow or torrential regime.
\end{itemize}
The Froude-number denominator $(gh)^{1/2}$ is the speed of an infinitesimal shallow-water surface wave.  As in gas dynamics, a channel flow can accelerate from subcritical to critical to supercritical flow and then return to subcritical flow through a shock called a \textit{hydraulic jump}, see \cite{Dasgupta2011} and references there in.\\
We are interested in the transition between different flow regimes when it occurs at a junction of a canals network.
On each canal the dynamics of water flow is described by the following system of \textit{one-dimensional shallow water equations}
\begin{equation}\label{eq:shallow_water}
\left( \begin{array}{c} h \\ hv \end{array}\right)_t + \left( \begin{array}{c} hv \\  hv^2 + \frac 12 g h^2 \end{array}\right)_x = 0.
\end{equation}
The quantity $q=hv$ is often called \textit{discharge} in shallow water theory, since it measures the rate of water past a point.
We write the system as:
\begin{equation}\label{eq:shallow_water_homog}
\partial_t u + \partial_x f(u) = 0,
\end{equation}
where
\begin{equation}\label{eq:shallow_water_flux}
u=\left(\begin{array}{c}
h \\ hv
\end{array}\right), \quad 
f(u)= \left(\begin{array}{c}
hv \\ h v^2 + \frac 12 g h^2
\end{array}\right) .
\end{equation}
For smooth solutions, these equations can be rewritten in
quasi-linear form 
\begin{equation}
\partial_t u + f^\prime(u)\partial_x u = 0,
\end{equation}
where the Jacobian matrix $f^\prime(u)$ is given by
\begin{equation}
f^\prime(u) = \left(\begin{array}{cc} 
0 & 1 \\ -v^2+gh & 2v
\end{array}\right).
\end{equation}
The eigenvalues of $f^\prime(u)$ are 
\begin{equation}\label{eq:eigenvalues}
\lambda_1=v-\sqrt{gh}, \quad \lambda_2=v+\sqrt{gh},
\end{equation}
with corresponding eigenvectors
\begin{equation}\label{eq:egenvectors}
r_1=\left(\begin{array}{cc} 1 \\ \lambda_1\end{array}\right),
\quad
r_2=\left(\begin{array}{cc} 1 \\ \lambda_2\end{array}\right).
\end{equation}
The shallow water equations are \textit{strictly hyperbolic} away from $h=0$ and both $1$-th and $2$-th fields are \textit{genuinely nonlinear} ($\nabla\lambda_j(u)\cdot r_j(u)\ne 0$, $j=1,2$).
Note that $\lambda_1$ and $\lambda_2$ can be of either sign, depending on the magnitude of $v$ relative to $\sqrt{gh}$, so depending on the Froude number.\\
Solutions to systems of conservation laws are usually constructed
via Glimm scheme of wave-front tracking \cite{Daf16}. 
The latter is based on the
solution to \textit{Rieman problems}:
\begin{equation}\label{eq:shallow_water_Rpb}
\left\{\begin{array}{l}
\partial_t u + \partial_x f(u) = 0,\\
\smallskip
u(x,0) = \left\{
\begin{array}{ll}
u_l & \mbox{ if } x<0,\\ u_r & \mbox{ if } x>0.
\end{array}
\right.
\end{array}\right.
\end{equation}
Here $u(x,0)=(h(x,0),q(x,0))$ and $u_l=(h_l,q_l)$ and $u_r=(h_r,q_r)$. The solution always consists of two waves, each of which is a shock or rarefaction,
thus we first describe these waves.
\begin{itemize}
\item[\textbf{(R)}]{Centered Rarefaction Waves}. Assume $u^+$ lies on the positive $i$-rarefaction curve through $u^-$, then we get
$$
u(x,t)=\left\{\begin{array}{ll}
	u^- & \mbox{ for } x<\lambda_i(u^-)t,\\
	R_i(x/t;u^-) & \mbox{ for } \lambda_i(u^-)t\leq x \leq \lambda_i(u^+)t,\\
	u^+ & \mbox{ for } x>\lambda_i(u^+)t,
	\end{array}\right.
$$
where, for the $1$-family
$$
R_1(\xi;u^-) := \left(\begin{array}{c}
\frac{1}{9}(v^-+2\sqrt{h^-}-\xi)^2\\
\frac{1}{27}(v^-+2\sqrt{h^-}+2\xi)(v^-+2\sqrt{h^-}-\xi)^2
\end{array}\right)
$$
for $\xi\in[v^+-\sqrt{h^-},v^-+2\sqrt{h^-})$, and for the second family
$$
R_2(\xi;u^-) := \left(\begin{array}{c}
\frac{1}{9}(-v^-+2\sqrt{h^-}-\xi)^2\\
\frac{1}{27}(v^--2\sqrt{h^-}+2\xi)(-v^-+2\sqrt{h^-}-\xi)^2
\end{array}\right)
$$ 
for $\xi\in[\lambda_2(u^-),\infty)$.

\item[\textbf{(S)}]{Shocks}. Assume that the state $u^+$ is connected to the right of $u^-$ by an $i$-shock, then calling $\lambda = \lambda_i(u^+,u^-)$ the Rankine-Hugoniot speed of the shock, the function
$$
u(x,t)=\left\{\begin{array}{ll}
	u^- & \mbox{ if } x<\lambda t\\
	u^+ & \mbox{ if } x>\lambda t
	\end{array}\right.
$$
provides a piecewise constant solution to the Riemann problem. For strictly hyperbolic systems, where the eigenvalues are distinct, we have that
$$
\lambda_i(u^+) < \lambda_i(u^-,u^+)<\lambda_i(u^-), \quad \lambda_i(u^-,u^+) = \frac{q^+-q^-}{h^+-h^-}.
$$
\end{itemize}

\section{The geometry of Lax and regime change curves}\label{sec:notations}
To determine a solution for problems on a network, we need to analyze
in detail the shape of shocks and rarefaction curves and, more generally,
of Lax curves (which are formed by joining shocks and rarefaction ones,
see \cite{Bressan}). We start fixing notations and illustrating the
shapes of curves.\\
For a given point $(h_0,v_0)$, we use the following notations for shocks and rarefaction curves:
\begin{equation}\label{SRdef}
\begin{array}{ll}
\mbox{for } h<h_0, & v = \Rcal_1(h_0,v_0;h) = v_0-2(\sqrt{gh}-\sqrt{gh_0});
\medskip\\
\mbox{for } h>h_0, &v = \Scal_1(h_0,v_0;h) = v_0-(h-h_0)\sqrt{g\frac{h+h_0}{2hh_0}};
\medskip\\
\mbox{for } h>h_0, &v = \Rcal_2(h_0,v_0;h) = v_0-2(\sqrt{gh_0}-\sqrt{gh});
\medskip\\
\mbox{for } h<h_0, & v = \Scal_2(h_0,v_0;h) = v_0-(h_0-h)\sqrt{g\frac{h+h_0}{2hh_0}}.
\end{array}
\end{equation}
Moreover, we define the inverse curves:
\begin{equation}\label{SRdef_m1}
\begin{array}{ll}
\mbox{for } h>h_0, & v = \Rcal^{-1}_1(h_0,v_0;h) = v_0+2(\sqrt{gh_0}-\sqrt{gh});
\medskip\\
\mbox{for } h<h_0, & v = \Scal^{-1}_1(h_0,v_0;h) = v_0+(h_0-h)\sqrt{g\frac{h+h_0}{2hh_0}}.
\end{array}
\end{equation}
Similarly, we set:
\begin{equation}\label{SRdef_m2}
\begin{array}{ll}
\mbox{for } h<h_0, & v = \Rcal^{-1}_2(h_0,v_0;h) = v_0+2(\sqrt{gh}-\sqrt{gh_0});
\medskip\\
\mbox{for } h>h_0, & v = \Scal^{-1}_2(h_0,v_0;h) = v_0+(h-h_0)\sqrt{g\frac{h+h_0}{2hh_0}}.
\end{array}
\end{equation}
We will also consider the regime transition curves:
the $1$-critical curve is given by
\begin{align}\label{eq:Cp}
v = \Ccal^+(h) = \sqrt{gh}
\end{align} 
and the $2$-critical curve by
\begin{align}\label{eq:Cm}
v = \Ccal^-(h) = -\sqrt{gh}.
\end{align}
In Figure \ref{fig:curve} we illustrate the shape of these curves.
\begin{figure}
\begin{tabular}{c}
\begin{overpic}
[width=0.9\textwidth]{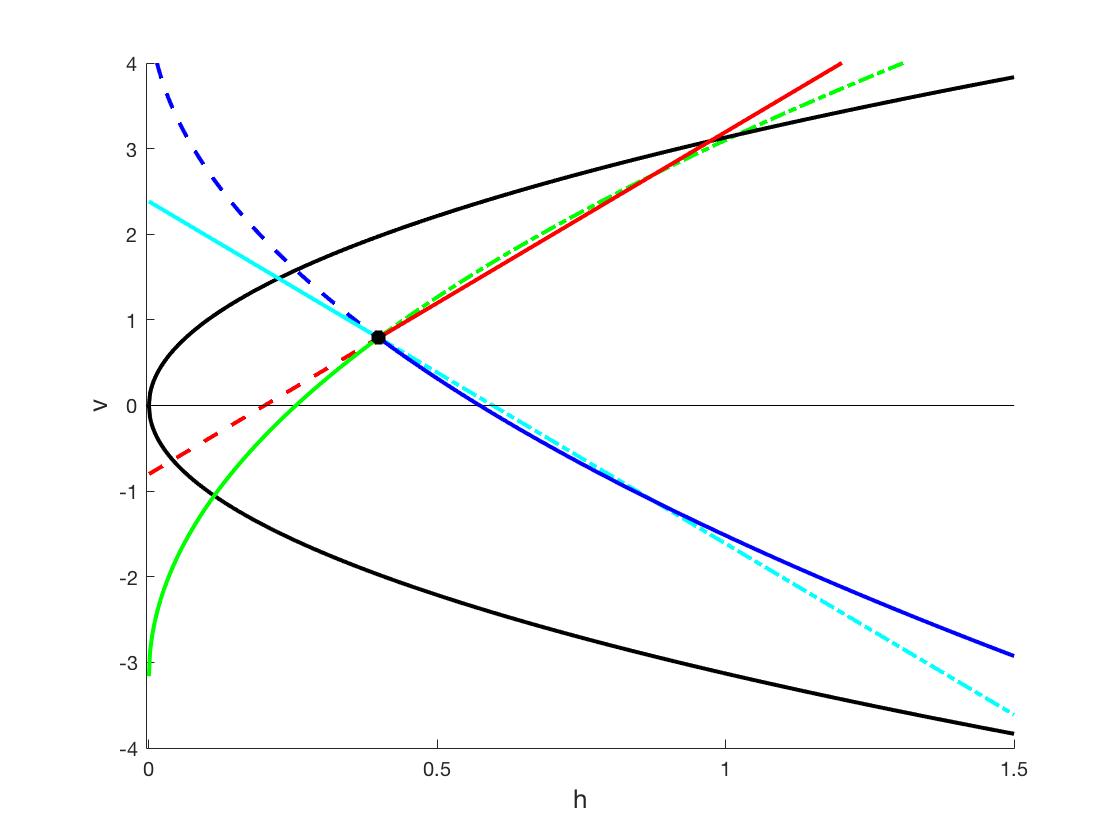}
\put(37,44){$(h_0,v_0)$}
\put(85,64){$C^+$}\put(65,11){$C^-$}
\put(15,51){$\Rcal_1$}\put(64,66){$\Rcal_2$}
\put(65,21){$\Rcal^{-1}_1$}\put(20,40){$\Rcal^{-1}_2$}
\put(85,19){$\Scal_1$}\put(20,60){$\Scal_1^{-1}$}
\put(22,32){$\Scal_2$}\put(77,70){$\Scal_2^{-1}$}
\end{overpic} 
\medskip\\
\begin{overpic}
[width=0.9\textwidth]{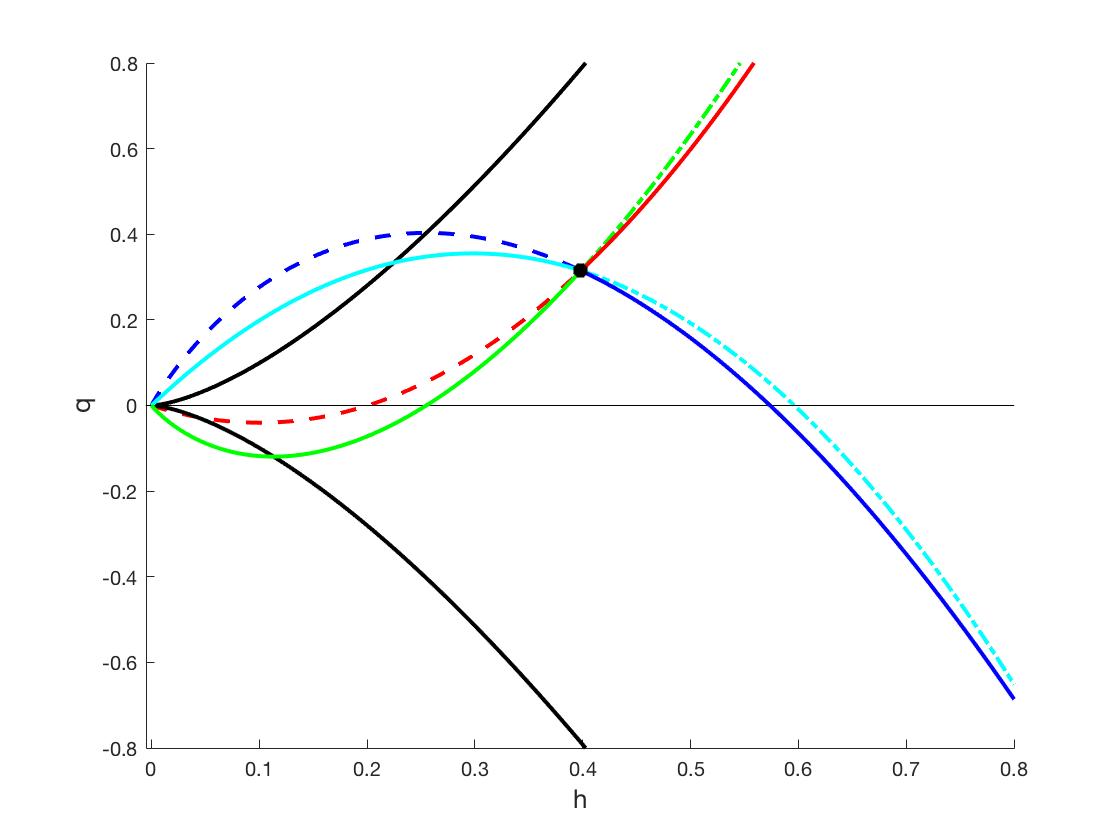}
\put(54,51){$(h_0,q_0)$}
\put(45,65){$\Ct^+$}\put(45,10){$\Ct^-$}
\put(40,49){$\tilde\Rcal_1$}\put(60,56){$\tilde\Rcal_2$}
\put(30,40){$\tilde\Rcal^{-1}_2$}\put(76,25){$\tilde\Scal_1$}
\put(85,23){$\tilde\Rcal_1^{-1}$}\put(20,50){$\tilde\Scal_1^{-1}$}
\put(44,40){$\tilde\Scal_2$}\put(59,65){$\tilde\Scal_2^{-1}$}
\end{overpic}
\end{tabular}
\caption{Shocks, rarefaction and critical curves \eqref{SRdef}-\eqref{eq:Cm} 
on the plane $(h,v)$ (up) and on the plane $(h,q)$ (down).}
\label{fig:curve}
\end{figure}
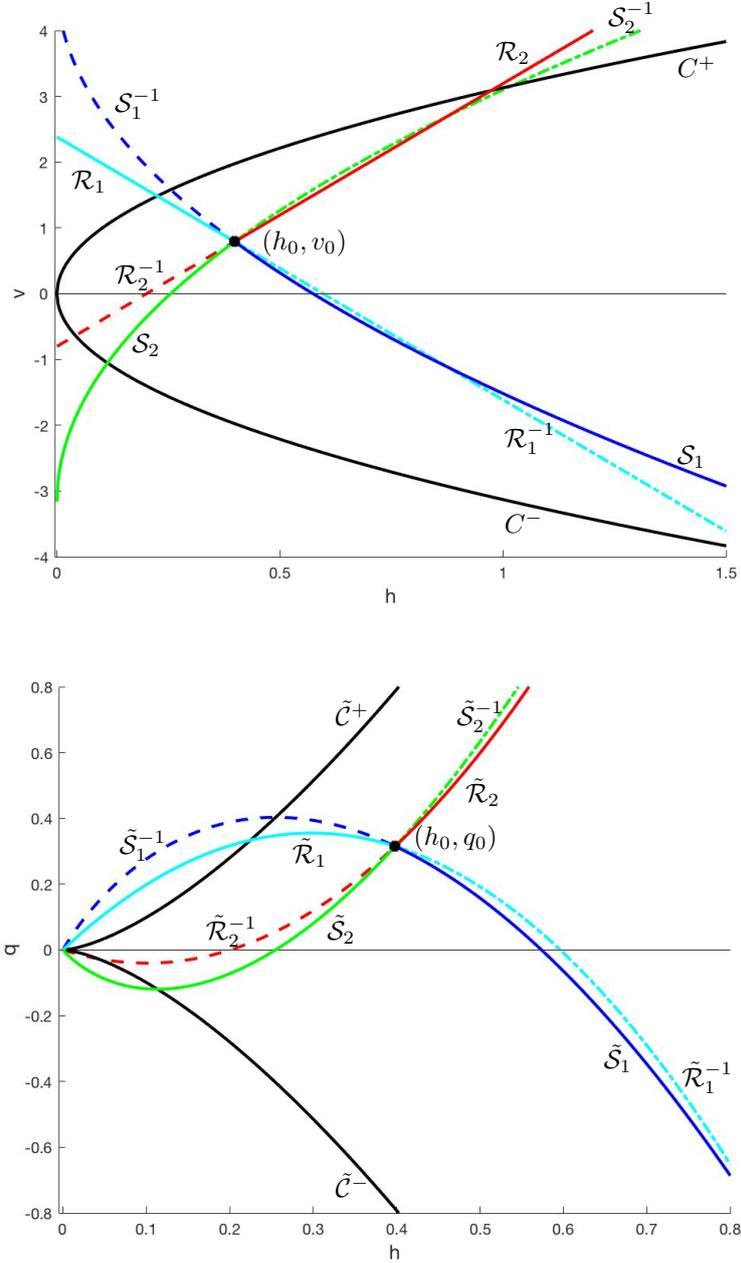
To construct a solution to a Riemann problem $(u_l,u_r)$, we define the Lax curves.
Given a left state $u_l=(h_l,q_l)$ the Lax curve is given by:
\begin{equation}\label{eq:phi_l}
\phi_l(h):=\Rcal_1(h_l,v_l;h)\cup\Scal_1(h_l,v_l;h).
\end{equation}
For the right state $u_r=(h_r,q_r)$,
we define the inverse Lax curve:
\begin{equation}\label{eq:phi_r}
\phi_r(h):=\Rcal_2^{-1}(h_r,v_r;h)\cup\Scal^{-1}_2(h_r,v_r;h).
\end{equation}

\begin{remark}\label{remark:std_sol}
The Riemann problem for shallow water equations \eqref{eq:shallow_water} with left state $u_l$ and right state $u_r$ has a unique solution if and only if 
the Lax curves $\phi_l(h)$ and $\phi_r(h)$ have a unique intersection. In that case, the intersection will be called the \textit{middle state} $u_m$.
As shown in Figure \ref{fig:phi}, the function $v=\phi_l(h)$ is strictly decreasing, unbounded and starting at the point $v_l+2\sqrt{gh_l}$ and, $v=\phi_r(h)$ is strictly increasing, unbounded, with minimum $v_r-2\sqrt{gh_r}$. Thus, the Riemann problem for shallow water has a unique solution in the region where
\begin{equation}\label{eq:region_cond}
v_l+2\sqrt{gh_l}\geq v_r-2\sqrt{gh_r}.
\end{equation}
\end{remark}
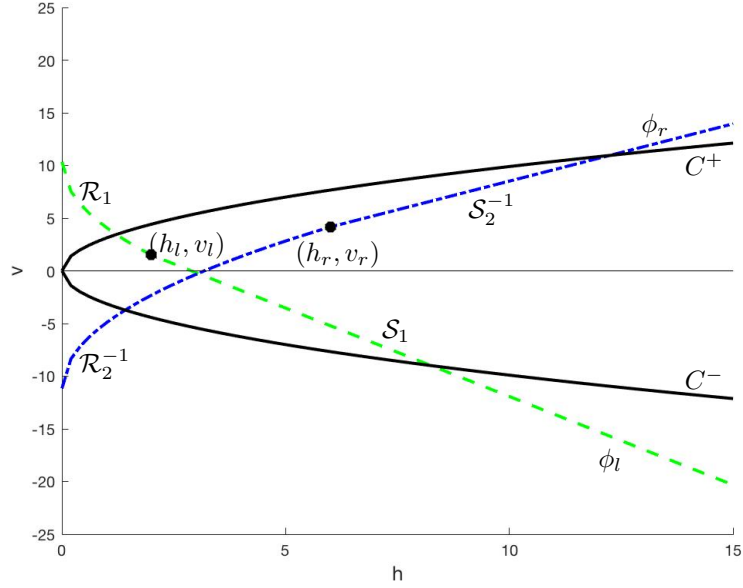
\begin{figure}
%
\begin{overpic}
[width=0.9\textwidth]{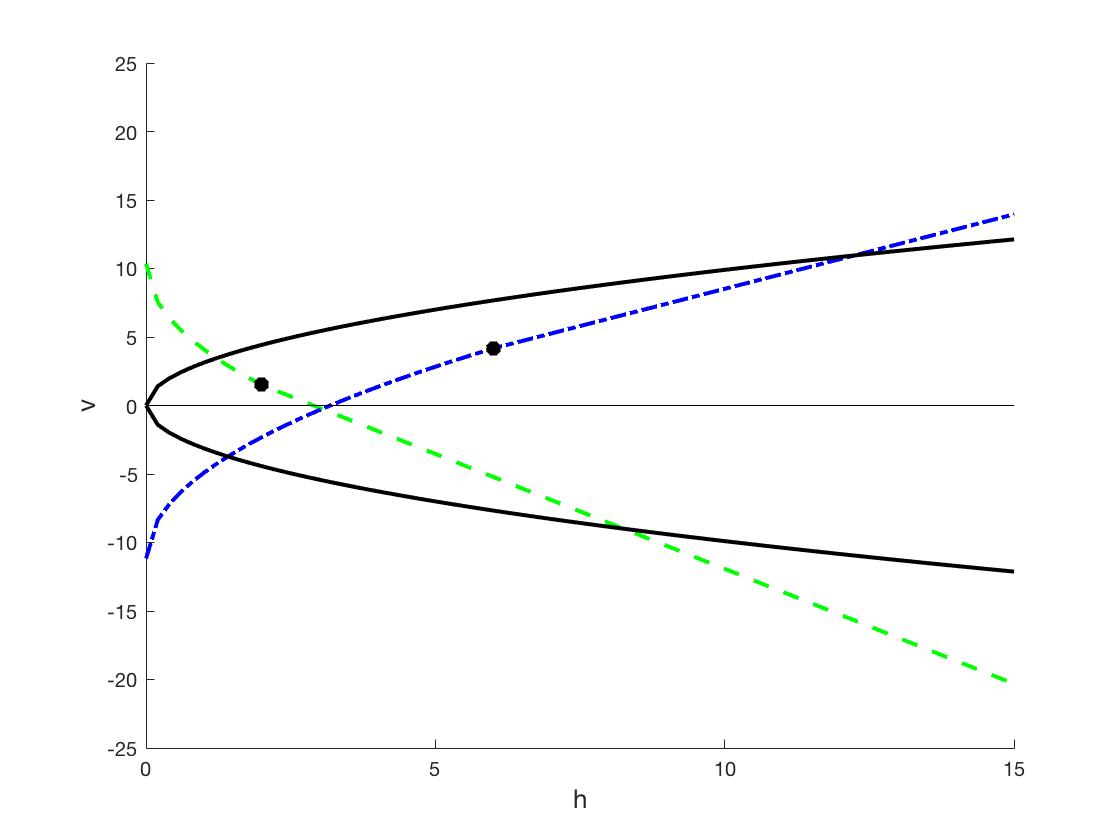}
\put(23,41){$(h_l,v_l)$}\put(40,40){$(h_r,v_r)$}
\put(85,50){$C^+$}\put(85,25){$C^-$}
\put(15,27){$\Rcal^{-1}_2$}\put(60,45){$\Scal^{-1}_2$}
\put(15,47){$\Rcal_1$}\put(50,31){$\Scal_1$}
\put(75,16){$\phi_l$}
\put(80,55){$\phi_r$}
\end{overpic} 

\caption{Graph of $\phi_l$ and $\phi_r$ defined in \eqref{eq:phi_l} and \eqref{eq:phi_r} respectively.}
\label{fig:phi}
\end{figure}
When working with $(h,q)$ variables we use the following notations
for Lax curves and regime transition curves:
\begin{align*}
\phit_l(h)=h\phi_l(h),\quad\phit_r(h)=h\phi_r(h)\quad\mbox{and}\quad\Ct^+(h)= h\Ccal^+(h),\quad\Ct^-(h)=h\Ccal^-(h).
\end{align*}
Moreover, for a given value $(h_i,v_i)$ (or $(h_i,q_i)$) we set
\begin{equation}\label{eq:froude}
\Fr_i= \frac{v_i}{\sqrt{gh_i}}, \quad\mbox{or}\quad \tilde \Fr_i = \frac{q_i}{h_i\sqrt{gh_i}}.
\end{equation}

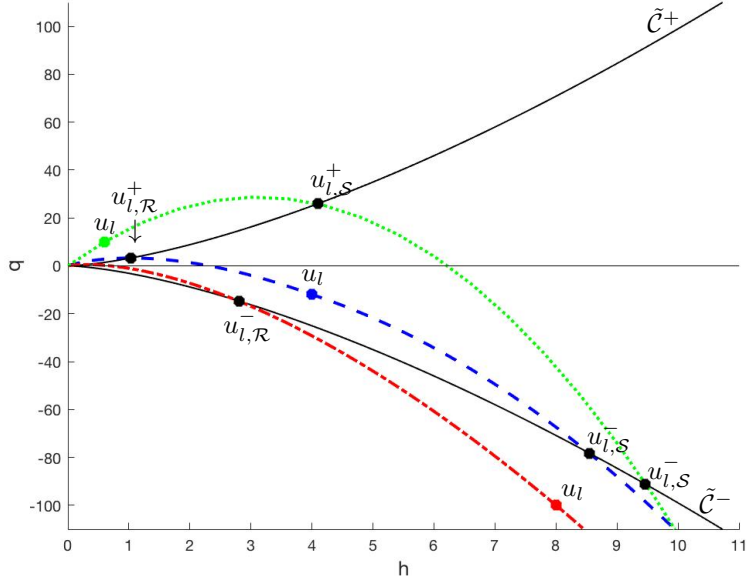
\begin{figure}

\begin{overpic}
[width=0.9\textwidth]{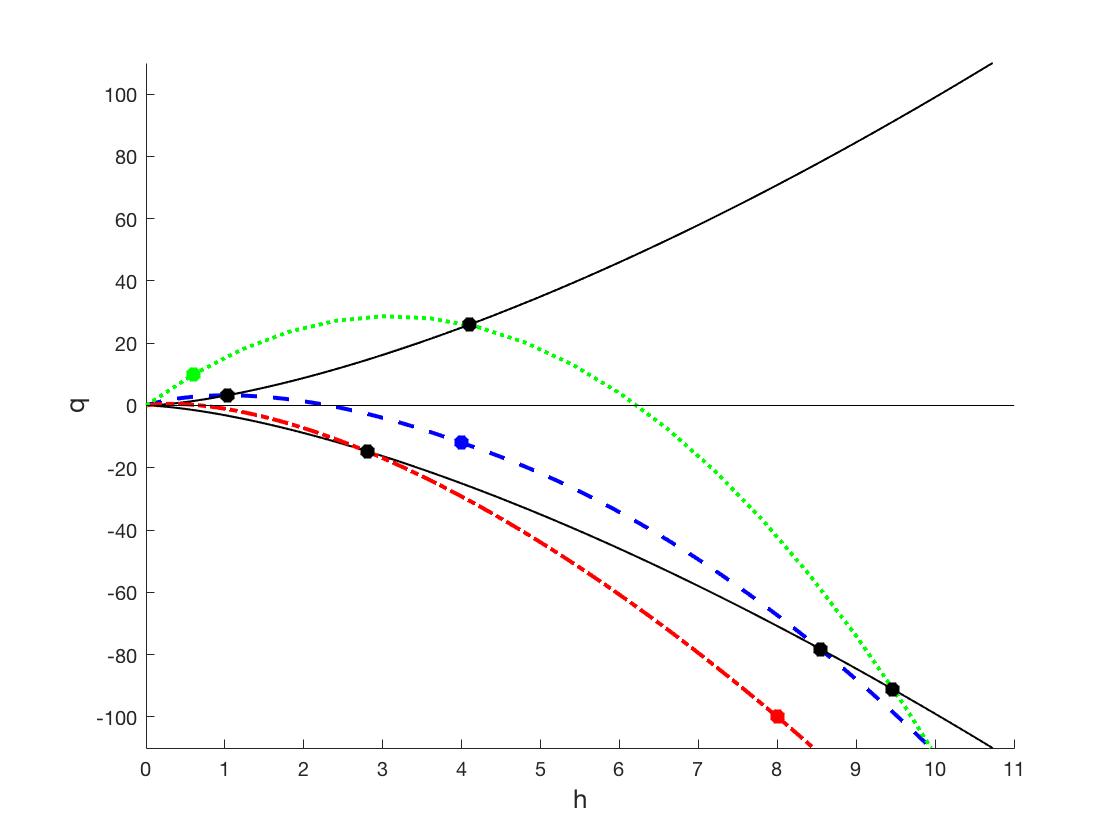}
\put(40,37){$u_l$}
\put(73,18){$u^-_{l,\Scal}$}
\put(18,46){$u^+_{l,\Rcal}$}
\put(20,42){$\huge\downarrow$}
\put(70,12){$u_l$}
\put(31,31){$u^-_{l,\Rcal}$}
\put(16,43){$u_l$}
\put(41,48){$u^+_{l,\Scal}$}
\put(80,14){$u^-_{l,\Scal}$}
\put(80,66){$\Ct^+$}
\put(86,10){$\Ct^-$}
\end{overpic}
%
\caption{Graph of $q=\phit_l(h)$ for different values of left state $u_l$ and its intersections with critical curves $q=\Ct^+(h)$ and $q=\Ct^-(h)$. The left state $u_l$ have been chosen such that: $F_l>1$ (dotted green line ), $|F_l|<1$ (blue dashed line) and $-2\leq F_l<-1$ (red dotted line).}
\label{fig:punti}
\end{figure}

\subsubsection{The Lax curve $\phit_l(h)$}\label{sec:phil} In this subsection we study in detail the properties of the function $q=\phit_l(h)$. For a given left state $u_l=(h_l,q_l)$,
$$
\phit_l(h) = \left\{\begin{array}{ll}
h\left(v_l+2\sqrt{gh_l}-2\sqrt{gh}\right), & 0<h\leq h_l,
\\
h\left(v_l-\sqrt{\frac{g}{2h_l}}(h-h_l)\sqrt{\frac{h+h_l}{h}}\right), & h>h_l
\end{array}\right.
$$
with 
$$
\lim_{h\rightarrow 0^+}\phit_l(h) = 0 \quad\mbox{and}\quad \lim_{h\rightarrow +\infty}\phit_l(h) = -\infty.
$$
By computing its first and second derivatives,
$$
\phit_l^\prime(h) = \left\{\begin{array}{ll}
v_l+2\sqrt{gh_l}-3\sqrt{gh}, & 0<h\leq h_l,
\\
v_l-\sqrt{\frac{g}{2h_l}}\left(\frac{4h^2+h_lh-h^2_l}{2\sqrt{h(h+h_l)}}\right), & h>h_l
\end{array}\right.
$$
and
$$
\phit_l^{\prime\prime}(h) = \left\{\begin{array}{ll}
-\frac{3}{2}\sqrt{\frac{g}{h}}, & 0<h\leq h_l,
\\
-\sqrt{\frac{g}{2h_l}}\left(\frac{8h^3+12h_lh+3h_l^2h+h_l^3}{4h(h+h_l)\sqrt{h(h+h_l)}}\right), & h>h_l,
\end{array}\right.
$$
we can conclude that $\phit_l\in C^2(]0,+\infty[)$. Specifically, in $]0,+\infty[$ $\phit_l^{\prime\prime}<0$, $\phit_l^\prime(h)$ is strictly decreasing and $\phit_l$ is a strictly concave function. As
$$ \phit_l^\prime(0)=v_l+2\sqrt{gh_l} \quad \mbox{ and } \quad \lim_{h\rightarrow+\infty}\phit_l(h) = -\infty$$
we investigate two different cases:\\
{\bf Case 1.} If $v_l\leq -2\sqrt{g h_l}$, $\phit_l$ is a strictly negative decreasing concave function and specifically
$\phit_l(h) < \Ct^-(h)$ for  $h>0$.
Therefore, in this case $\phit_l$ never intersects the critical curves $\Ct^+$ and $\Ct^-$.\\
{\bf Case 2.}
If $v_l>-2\sqrt{g h_l}$, the function $\phit_l$ admits a maximum point $h_l^{max}$, so it is increasing in $(0,h_l^{max})$ an decreasing in $(h^{max}_l,+\infty)$. In this case the function $\phit_l$ intersects the two critical curves $\Ct^+$ and $\Ct^-$. The intersection points vary with the choice of the left state $u_l$, see Figure \ref{fig:punti}. To compute this points, we distinguish the following subcases: $u_l$ is such that $-2\leq\Fr_l\leq 1$ or such that $\Fr_l>1$.\\
{\bf Case 2.1.}
For $-2<\Fr_l \leq 1$, it is the rarefaction portion of $\phit_l$ to intersect the 1-critical curve $\Ct^+$ at $u^+_{l,\Rcal} = (h^+_{l,\Rcal}, \Ct^+(h^+_{l,\Rcal}))$ with
\begin{equation}\label{uplr}
h^+_{l,\Rcal} = \frac{1}{9g}\left(v_l+2\sqrt{gh_l}\right)^2,
\end{equation}
and we have that the maximum point $h_l^{max}$ is such that $\phit_l(h_l^{max})\equiv \Ct^+(h_l^{max})$, i.e. $h_l^{max}\equiv h^+_{l,\Rcal}$.
In the special case $\Fr_l=1$ we get $h_l^{max}\equiv h_l$.
Moreover,
$$
	\phit_l(h)=0 \Leftrightarrow h =\frac{1}{4g}(v_l+2\sqrt{gh_l})^2. 
$$ 
Notice that for $-2<\Fr_l\leq -1$
the rarefaction portion of $\phit_l$ intersects $\Ct^-$ at 
$u^-_{l,\Rcal} = (h^-_{l,\Rcal}, \Ct^-(h^-_{l,\Rcal}))$ with
\begin{equation}\label{umlr}
h^-_{l,\Rcal} = \frac{1}{g}\left(v_l+2\sqrt{gh_l}\right)^2;
\end{equation}
while for $-1<\Fr_l \leq 1$, the shock portion of $\phit_l$ intersects  $\Ct^-$ at $u^-_{l,\Scal} = (h^-_{l,\Scal}, \Ct^-(h^-_{l,\Scal}))$ with
$h^-_{l,\Scal}$ given by the following condition:
\begin{equation}\label{umls}
v_l-(h^-_{l,\Scal}-h_l)\sqrt{g\frac{h^-_{l,\Scal}+h_l}{2h_lh^-_{l,\Scal}}}+\sqrt{gh^-_{l,\Scal}} = 0;
\end{equation}
{\bf Case 2.2.}
For $\Fr_l >1$, the maximum value of $\phit_l$ is reached by the shock portion and so $h_l^{max}>h_l$. The shock portion intercepts the critical curve $\Ct^+$ at $u^+_{l,\Scal} = (h^+_{l,\Scal}, \Ct^+(h^+_{l,\Scal}))$ with
\begin{equation}\label{upls}
h^+_{l,\Scal} \quad \mbox{such that} \quad
v_l-(h^+_{l,\Scal}-h_l)\sqrt{g\frac{h^+_{l,\Scal}+h_l}{2h_lh^+_{l,\Scal}}}-\sqrt{gh^+_{l,\Scal}} = 0;
\end{equation}
and the curve $\Ct^-$ at $u^-_{l,\Scal}$ defined in \eqref{umls}. 

Notice that for $h\geq h_l$ the equation 
$$
q_l=\phit_l(h) = h \Scal_1(h_l,v_l;h)
$$
has two solutions: 
\begin{equation}\label{eq:hstarl}
h=h_l \quad\mbox{and}\quad h=h^*_l=\frac{h_l}{2} \left(-1+\sqrt{1+8\Fr_l^2}\right). 
\end{equation}
Moreover, for $\Fr_l>1$, we have $h_l^*>h^+_l$, where the height $h^+_l=(q_l^2/g)^\frac 13$ is given by the intersection between $\Ct^+$ and the horizontal line $q=q_l$. That is: 
$$
(\frac{q_l^2}{g})^{\frac 13} < \frac{h_l}{2}(-1+\sqrt{1+8\Fr_l^2}), \quad \mbox{for}\quad \Fr_l>1.
$$
Indeed, using the relation  $q_l^2 = gh_l^3\Fr_l^2$,
if $\Fr_l>1$ then 
$
2\Fr_l^{\frac{4}{3}} - \Fr_l^{\frac{2}{3}}-1 > 0.
$
So, for $(h_l,q_l)$ supercritical, the point $(h^*_l,q_l)$ is subcritical, i.e. $|\Fr^*_l|<1$.

\subsubsection{The Lax curve $\phit_r(h)$}\label{sec:phir}
Here we study the properties of the function $\phit_r(h)$, given by:
$$
\phit_r(h) = \left\{\begin{array}{ll}
h\left(v_r-2\sqrt{gh_r}+2\sqrt{gh}\right), & 0<h\leq h_r,
\\
h\left(v_r+\sqrt{\frac{g}{2h_r}}(h-h_r)\sqrt{\frac{h+h_r}{h}}\right), & h>h_r.
\end{array}\right.
$$
By straightforward computations we get its derivatives:
$$
\phit_r^\prime(h) = \left\{\begin{array}{ll}
v_r-2\sqrt{gh_r}+3\sqrt{gh}, & 0<h\leq h_r,
\\
v_r+\sqrt{\frac{g}{2h_r}}\left(\frac{4h^2+h_rh-h^2_r}{2\sqrt{h(h+h_r)}}\right), & h>h_r.
\end{array}\right.
$$
$$
\phit_r^{\prime\prime}(h) = \left\{\begin{array}{ll}
\frac{3}{2}\sqrt{\frac{g}{h}}, & 0<h\leq h_r,
\\
\sqrt{\frac{g}{2h_r}}\left(\frac{8h^3+12h_rh+3h_r^2h+h_r^3}{4h(h+h_r)\sqrt{h(h+h_r)}}\right), & h>h_r.
\end{array}\right.
$$
Then, $\phit_r(h)\in C^2(]0,+\infty[)$ and 
in $]0,+\infty[$ $\phit_l^{\prime\prime}>0$, $\phit_l^\prime(h)$ is strictly increasing and $\phit_l$ is a strictly convex function. As
$$ \phit_r^\prime(0)=v_r-2\sqrt{gh_r} \quad \mbox{ and } \quad \lim_{h\rightarrow+\infty}\phit_r(h) = +\infty$$
we investigate two different cases:\\
{\bf Case 1.} If $v_r\geq 2\sqrt{gh_r}$, $\phit_r$ is a strictly positive increasing convex function and specifically
$\phit_r(h) > \Ct^+(h)$, for $h>0$.
Therefore, $\phit_r$ never intersects the critical curves $\Ct^+$ and $\Ct^-$.\\
{\bf Case 2.} If $v_r< 2\sqrt{g h_l}$, the function $\phit_r$ admits a minimum point $h_r^{min}$. In this case the function $\phit_r$ intersects the two critical curves $\Ct^+$ and $\Ct^-$.  As done before we distinguish two subcases:\\
{\bf Case 2.1.} For $-1<\Fr_r \leq 2$, the rarefaction portion of $\phit_r$  intersects the 2-critical curve $\Ct^-$ at $u^-_{r,\Rcal} = (h^-_{r,\Rcal}, \Ct^-(h^-_{r,\Rcal}))$, with
\begin{equation}\label{umrr}
\begin{array}{l}
h^-_{r,\Rcal} = \frac{1}{9g}\left(-v_r+2\sqrt{gh_r}\right)^2
\end{array}
\end{equation}
and we have that the minimum point $h_r^{min}$ is such that $\phit_r(h_r^{min})\equiv \Ct^-(h_r^{min})$, i.e. $h_r^{min}\equiv h^-_{r,\Rcal}$.
In the special case $\Fr_r=1$ we get $h_r^{min}\equiv h_r$.
Moreover,
$$
	\phit_r(h)=0 \Leftrightarrow h =\frac{1}{4g}(-v_r+2\sqrt{gh_r})^2. 
$$ 
For $1<\Fr_r\leq 2$ the rarefaction portion of $\phit_r$ intersects $\Ct^+$ at $u^+_{r,\Rcal} = (h^+_{r,\Rcal}, \Ct^+(h^+_{r,\Rcal}))$, with
\begin{equation}\label{uprr}
\begin{array}{l}
h^+_{r,\Rcal} = \frac{1}{g}\left(-v_r+2\sqrt{gh_r}\right)^2,
\end{array}
\end{equation}
while for $-1<\Fr_r \leq 1$, the shock portion of $\phit_r$ intersects  $\Ct^+$ at $u^+_{r,\Scal} = (h^+_{r,\Scal}, \Ct^+(h^+_{r,\Scal}))$, with
$h^+_{r,\Scal}$ such that
\begin{equation}\label{uprs}
v_r+(h^+_{r,\Scal}-h_r)\sqrt{g\frac{h^+_{r,\Scal}+h_r}{2h_rh^+_{r,\Scal}}}+\sqrt{gh^+_{r,\Scal}} = 0.
\end{equation}
{\bf Case 2.2.} 
For $\Fr_r <-1$, the minimum value of $\phit_r$ is reached by the shock portion and so $h_r^{min}>h_r$. The shock portion intercepts the critical curves $\Ct^-$ at $u^-_{r,\Scal}=(h^-_{r,\Scal},\Ct^-(h^-_{r,\Scal}))$, with
$$
h^-_{r,\Scal} \mbox{ such that } v_r+(h_r-h^+_{r,\Scal})\sqrt{g\frac{h^-_{r,\Scal}+h_r}{2h_rh^-_{r,\Scal}}}+\sqrt{gh^-_{r,\Scal}} = 0
$$
and curve $\Ct^+$ at $u^+_{r,\Scal}$ given by \eqref{uprs}.

Notice that the equation 
$$
q_r=\phit_r(h)=h\Scal_2^{-1}(h_r,v_r;h), \quad h\geq h_r,
$$
has two solutions: 
\begin{equation}\label{eq:hstarr}
h=h_r \quad\mbox{and}\quad h=h^*_r=\frac{h_r}{2} \left(-1+ \sqrt{1+8\Fr_r^2}\right). 
\end{equation}
Moreover, for $\Fr_r<-1$ we have $h_r^*>h^-_r=(q_r^2/g)^\frac 13$, where the height $h_r^-$ is given by the intersection between $\Ct^-$ and the horizontal line $q=q_r$. So, for $(h_r,q_r)$ supercritical, the point $(h^*_r,q_r)$ is subcritical.

\section{The half Riemann problems}\label{sec:half-Riemann}

\subsection{Left-half Riemann problem}[The case of an incoming canal]\label{sec:leftRP}
We fix a left state and we look for the right states attainable by waves of 
non-positive speed.\\
Fix $u_l=(h_l,q_l)$, we look for the set $\Ncal(u_l)$ of points $\hat u=(\hat h,\hat q)$ such that the solution to the Riemann problem
\begin{equation}\label{Rleft}
\left\{
\begin{array}{l}
\partial_t u + \partial_x f(u) = 0,\\
u(x,0)=\left\{\begin{array}{ll}
	u_l & \mbox{ if } x<0\\
	\hat u & \mbox{ if } x>0
	\end{array}\right.
\end{array}
\right.
\end{equation}
contains only waves with non-positive speed.
We distinguish three cases: 
\begin{itemize}
\item \textit{Case A}: the left state $u_l$ is such that $|\tilde\Fr_l| < 1$;
\item \textit{Case B}: the left state $u_l$ is such that $\tilde\Fr_l > 1$;
\item \textit{Case C}: the left state $u_l$ is such that $\tilde\Fr_l <- 1$.
\end{itemize}
\subsubsection{Case A} For this case we refer to Figure \ref{fig:in_A}.
\begin{figure}
\begin{center}
\begin{overpic}
[width=0.9\textwidth]{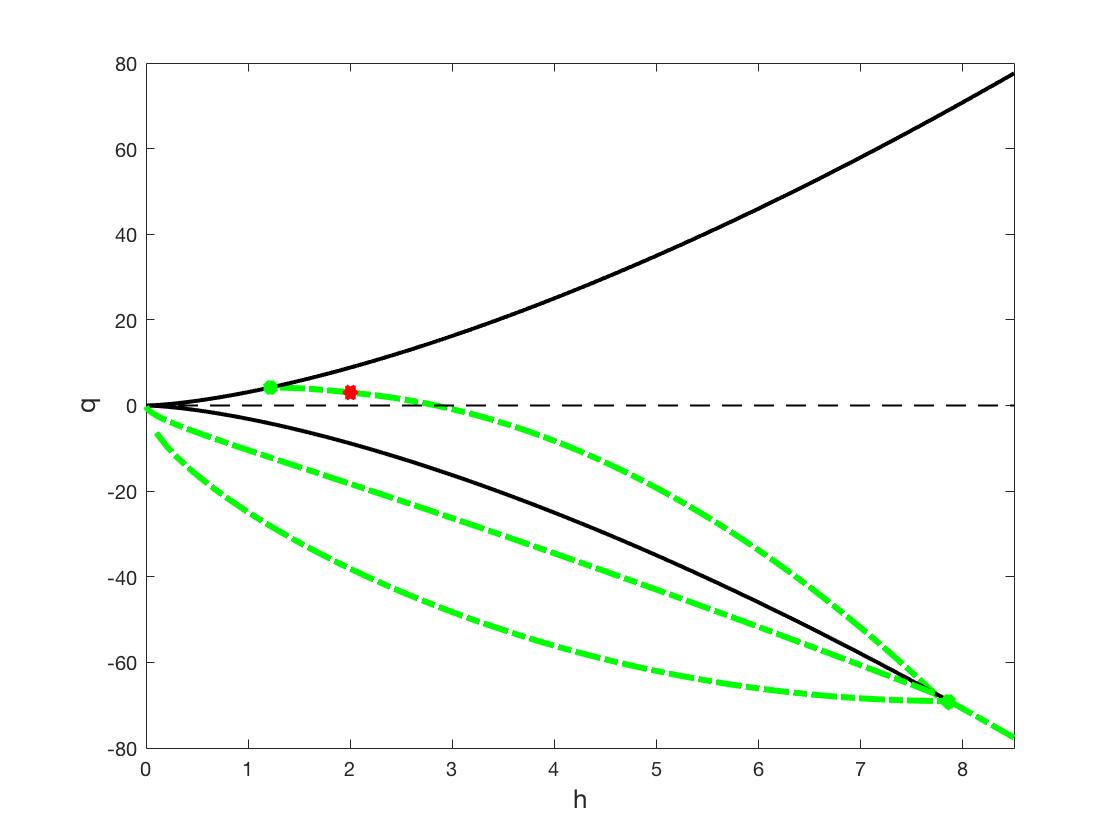}
\put(22,43){$u^+_{l,\Rcal}$}\put(84,13){$u^-_{l,\Scal}$}
\put(33,41){$u_l$}\put(25,13){$\Ical_2^A$}\put(45,22){$\Ical_3^A$}\put(66,34){$\Ical_1^A$}\put(62,31){$\huge\swarrow$}
\put(85,8){$\Ct^-$}
\put(40,17){\makebox(0,0){\rotatebox{-18}{$q=\tilde\Scal_2(u^-_{l,\Scal};h)$}}}
\end{overpic}
\end{center}
\caption{Left-half Riemann problem, Section \ref{sec:leftRP}. Region $\Ncal^A(u_l)=\Ical^A_1\bigcup\Ical^A_2\bigcup\Ical^A_3$ defined by \eqref{eq:IA1}-\eqref{eq:IA3}. Following our notation $\tilde\Scal_2(u^-_{l,\Scal};h)=h\Scal_2(h^-_{l,\Scal},\Ccal^-(h^-_{l,\Scal});h)$. }\label{fig:in_A}
\end{figure}
We identify the set $\Ncal^A(u_l)$ as the union of three regions $\Ical^A_1$, $\Ical^A_2$ and $\Ical^A_3$ defined in the plane $(h,q)$. 
The first region is identified by all points that belong to the curve $\phit_l(h)$ such that $\Ct^-(h)\leq\phit_l(h)\leq \Ct^+(h)$, i.e. 
\begin{equation}\label{eq:IA1}
\Ical^A_1 = \left\{(\hat h,\hat q) : h^+_{l,\Rcal} \leq \hat h \leq h^-_{l,\Scal}, \hat q=\phit_l(\hat h)\right\},
\end{equation}
where the points $h^+_{l,\Rcal}$ and $h^-_{l,\Scal}$ are given in \eqref{uplr} and \eqref{umls} respectively.
The second region is defined as follows
\begin{equation}\label{eq:IA2}
\begin{array}{lcl}
\Ical^A_2 &=& \left\{(\hat h,\hat q) : 0<\hat h\leq h^-_{l,\Scal}, \ \hat q\leq \hat h\Scal_2(h^-_{l,\Scal}, \Ccal^-(h^-_{l,\Scal}); \hat h)\right\}
\\
&&\bigcup\left\{(\hat h,\hat q): \hat h>h^-_{l,\Scal}, \ \hat q \leq \Ct^-(\hat h)\right\}.
\end{array}
\end{equation}
The last region $\Ical^A_3$ is defined by the set of all possible right states $\hat u$ that can be connected by a $2$-shock with non-positive speed to an intermediate state $u_m$ lying on $\phit_l(h)$ curve such that 
$|\tilde\Fr_m|\leq 1$ and
$$
\lambda(u_m,\hat u) = \frac{q_m-\hat q}{h_m-\hat h}\leq 0.
$$
To define this region, we have to look for values $q=hS_2(h_m,v_m;h)$ as $h<h_m$ such that $q\geq q_m$. That is,
$$ q_m-q=(h_m-h)\left(v_m+\sqrt{\frac{g}{2h_m}}\sqrt{h(h+h_m)}\right)\leq 0, \quad h<h_m. $$
This inequality is verified for $-1\leq \Fr_m <0$ and for all $h\leq h^*_m$ with $h^*_m$ given by
$$
h^*_m = \frac{h_m}{2} \left(-1 +\sqrt{1+8\Fr_m^2}\right).
$$
We obtain (see Figure \ref{fig:in_A}),
\begin{equation}\label{eq:IA3}
\begin{array}{lcl}
\Ical^A_3 &= & \left\{(\hat h,\hat q) : \mbox{ for all } (h_m,q_m) \mbox{ which vary on } \phit_l  \mbox{ such that } -1\leq \tilde\Fr_m < 0, \right.
\\
\smallskip
&&\left. \ \  0<\hat h\leq h^*_m, \ \hat q=\hat h\Scal_2(h_m,v_m;\hat h)\right\}.
\end{array}
\end{equation}

\subsubsection{Case B} For this case we refer to Figure \ref{fig:in_B}. 
\begin{figure}
\begin{center}
\begin{overpic}
[width=0.9\textwidth]{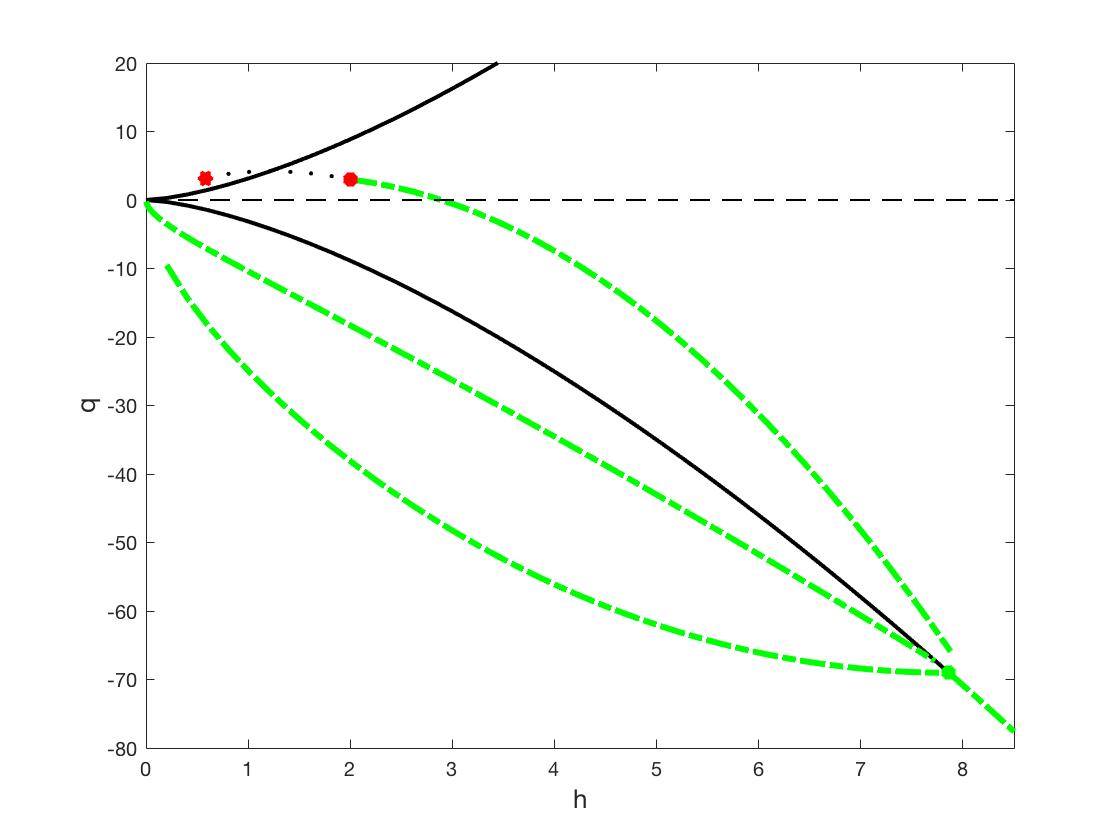}
\put(15,60){$u_{l}$}\put(32,60){$u_l^*$}
\put(86,15){$u^{-}_{l,\Scal}$}
\put(30,15){$\Ical_2^{*,A}$}\put(45,30){$\Ical_3^{*,A}$}\put(66,49){$\Ical_1^{*,A}$}\put(62,46){$\huge\swarrow$}
\put(84,10){$\Ct^-$}
\put(40,25){\makebox(0,0){\rotatebox{-27}{$q=\tilde\Scal_2(u^{*,-}_{l,\Scal};h)$}}}
\end{overpic}
\end{center}
\caption{Left-half Riemann problem, Section \ref{sec:leftRP}. Region $\Ncal^B(u_l)=\Ical^{*,A}_1\bigcup\Ical^{*,A}_2\bigcup\Ical^{*,A}_3$ given in \eqref{eq:NB}.}\label{fig:in_B}
\end{figure}
It is always possible to connect the left value 
$u_l$ to a value $u^*_l$ by a $1$-shock with \textit{zero} speed (vertical shock). Specifically, we set
$h^*_l>h_l$ such that $q^*_l=q_l$, i.e. 
$$
h^*_l =  \frac 12 \left(-1 +\sqrt{1+8\Fr_l^2}\right) h_l
$$
as previously computed in \eqref{eq:hstarl}. Moreover, as previously observed at the end of subsection \ref{sec:phil} the value $(h^*_l,q_l)$ falls in the subcritical region, then 
\begin{equation}\label{eq:NB}
\Ncal^B(u_l)=\Ncal^A(u_l)\setminus \left\{\hat u=(\hat h,\hat q): h^+_{l,\Scal}\leq \hat h\leq h_l^*,\ \hat q=\phit_l(\hat h)\right\},
\end{equation}
where $\Ncal^A(u_l)$ falls under the previous \textit{Case A}.

\subsubsection{Case C} For this case we refer to Figure \ref{fig:in_C}.
\begin{figure}
\begin{center}
\begin{overpic}
[width=0.9\textwidth]{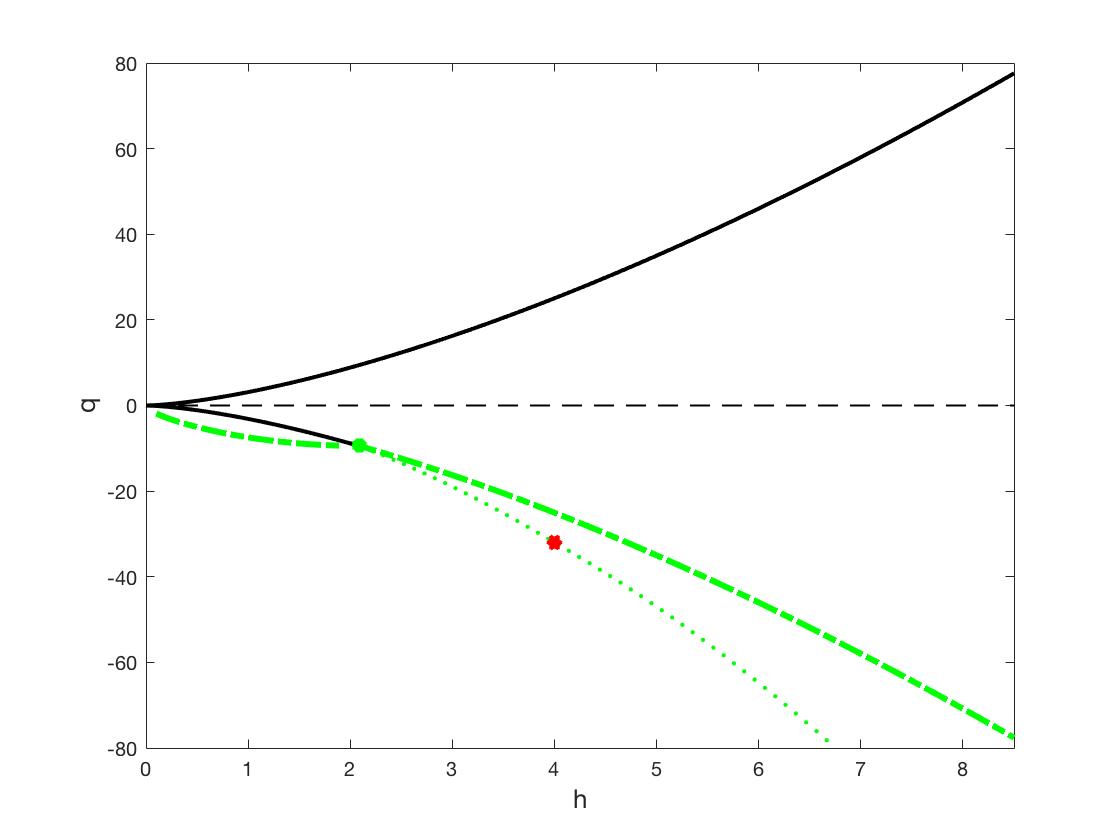}
\put(32,36){$u^-_{l,\Rcal}$}
\put(46,25){$u_l$}
\put(30,20){$\Ncal_C(u_l)$}
\put(80,10){$\Ct^-$}
\put(23,33){\makebox(0,0){\rotatebox{-8}{$q=\tilde\Scal_2(u^{-}_{l,\Rcal};h)$}}}
\end{overpic}
\end{center}
\caption{Left-half Riemann problem, Section \ref{sec:leftRP}. Region $\Ncal^C(u_l)$ bounded by $q=\tilde\Scal_2(u^{-}_{l,\Rcal};h)$ and $q=\Ct^-(h)$ as defined in \eqref{eq:NC}.}\label{fig:in_C}
\end{figure}
We have that: if $\Fr_l \leq -2$ then 
$$\Ncal^C(u_l)=\left\{(\hat h,\hat q): \hat h> 0,\  \hat q < \Ct^-(\hat h)\right\};$$ 
otherwise if $-2<\Fr_l<1$ the intersection point $u^-_{l,\Rcal}$ between $h\Rcal_1(h_l,v_l;h)$ and $\Ct_-(h)$ given in \eqref{umlr} is different from zero and defines the admissible region 
\begin{equation}\label{eq:NC}
\begin{array}{lcl}
\Ncal^C(u_l)&=&\left\{(\hat h,\hat q): 0< \hat h\leq h^-_{l,\Rcal},\ q< \hat h S_2(h^-_{l,\Rcal},v^-_{l,\Rcal};\hat h)\right\}
\\ \smallskip
&&\bigcup \left\{(\hat h,\hat q): \hat h>h^-_{l,\Rcal},\ \hat q< \Ct^-(\hat h)\right\}.
\end{array}
\end{equation}

\subsection{Right-half Riemann problem}[The case of an outgoing canal]\label{sec:rightRP}
We fix a right state and we look for the left states attainable by waves of non-negative speed. For sake of space the figures
illustrating these cases will be postponed to the Appendix.\\
Fix $u_r=(h_r,q_r)$, we look for the set $\Pcal(u_r)$ of points $\tilde u=(\tilde h,\tilde q)$ such that the solution to the Riemann problem
\begin{equation}\label{Rright}
\left\{
\begin{array}{l}
\partial_t u + \partial_x f(u) = 0,\\
u(x,0)=\left\{\begin{array}{ll}
	\tilde u & \mbox{ if } x<0\\
	u_r & \mbox{ if } x>0
	\end{array}\right.
\end{array}
\right.
\end{equation}
contains only waves with non-negative speed.
As in the previous case we identify three cases:
\begin{itemize}
\item \textit{Case A}: the right value $u_r$ is such that $|\tilde\Fr_r|<1$.

\item \textit{Case B}: the right value $u_r$ is such that $\tilde\Fr_r>1$

\item \textit{Case C}: the right value $u_r$ is such that $\tilde\Fr_r<-1$.
\end{itemize}
\subsubsection{Case A} For this case we refer to Figure \ref{fig:out_A} in the Appendix. We identify the set $\Pcal^A(u_r)$ as the union of three regions $\Ocal^A_1$, $\Ocal^A_2$ and $\Ocal^A_3$ defined in the plane $(h,q)$. The first region is defined by all points that belong to the curve $\phit_r(h)$ such that $\Ct^-(h)\leq\phit_r(h)\leq\Ct^+(h)$, i.e.
\begin{equation}\label{eq:OA1}
\Ocal^A_1=\left\{(\tilde h,\tilde q): h^-_{r,\Rcal}\leq \tilde h\leq h^+_{r,\Scal},\ \tilde q=\phit_r(\tilde h)\right\},
\end{equation}
where the points $h^-_{r,\Rcal}$ and $h^+_{r,\Scal}$ are given in \eqref{umrr} and \eqref{uprs} respectively.
\\
The second region is such that
\begin{equation}\label{eq:OA2}
\begin{array}{lcl}
\Ocal^A_2&=&\left\{(\tilde h,\tilde q): 0<\tilde h\leq h^+_{r,\Scal},\ \tilde q\geq \tilde h\Scal_1^{-1}(h^+_{r,\Scal},v^+_{r,\Scal};\tilde h)\right\}
\smallskip\\
&&\bigcup\left\{ \tilde h\geq h^+_{r,\Scal},\ \tilde q\geq \Ct^+(\tilde h)\right\}.
\end{array}
\end{equation}
The third region is defined by the set of all possible left states $\tilde u$ that can be connected by a 1-shock with non-negative speed to a middle state $u_m$ lying on $\phit_r(h)$ curve such that $|\tilde\Fr_m|\leq 1$ and
$$
\lambda(u_m,\tilde u) = \frac{q_m-\tilde q}{h_m-\tilde h}\geq 0.
$$
To define this region we have to look for values $q=h\Scal_1^{-1}(h_m,v_m;h)$ for $h<h_m$ such that $q_m\geq q$. Then, following the same reasoning  done for the left-half Riemann problem we get
\begin{equation}\label{eq:OA3}
\begin{array}{lcl}
\Ocal^A_3 &=& \left\{(\tilde h,\tilde q) : \mbox{ for all } (h_m,q_m) \mbox{ which vary on } \phit_r \mbox{ such that } 0<\tilde\Fr_m \leq 1\right.
\\
\smallskip 
&&\left. \ \ 0<\tilde h\leq h^*_m, \ \tilde q=\tilde h\Scal_1^{-1}(h_m,v_m;\tilde h)\right\}.
\end{array}
\end{equation}

\subsubsection{Case B} For this case we refer to Figure \ref{fig:out_B}. If $\Fr\geq 2$ then
$$
\Pcal^B(u_r)=\left\{(\tilde h,\tilde q):\ \tilde h\geq 0,\ \tilde q\geq \Ct^+(\tilde h)\right\};
$$
otherwise, if $1\leq \Fr <2$ the intersection point $u^+_{r,\Rcal}$ between $\Rcal_2^{-1}(u_r;h)$ and $\Ccal^+(h)$, given in \eqref{umrr} is different from zero and defines the admissible region
\begin{equation}\label{eq:PB}
\begin{array}{lcl}
\Pcal^B(u_r)&=&\left\{(\tilde h,\tilde q):\ 0< \tilde h \leq h^+_{r,\Rcal}, \ \tilde q\geq \tilde h\Scal_1^{-1}(u^-_{r,\Rcal};\tilde h)\right\}
\smallskip\\
&&\bigcup\left\{(\tilde h,\tilde q), \tilde h>h^-_{r,\Rcal},\ \tilde q\geq \Ct^+(\tilde h) \right\}.
\end{array}
\end{equation}

\subsubsection{Case C} For this case we refer to Figure \ref{fig:out_C}. 
It is always possible to connect the right value 
$u_r$ to a value $u^*_r$ by a $2$-shock with \textit{zero} speed (vertical shock). Specifically, we set
$h^*_r>h_r$ such that $q^*_r=q_r$, i.e. 
$$
h^*_r =  \frac 12 \left(-1 +\sqrt{1+8\Fr_r^2}\right) h_r
$$
as done in \eqref{eq:hstarr}. Moreover, as previously observed at the end of subsection \ref{sec:phir}, the point $(h^*_r,q_r)$ is subcritical and then 
\begin{equation}\label{eq:PC}
\Pcal^C(u_r)=\Pcal^A(u^*_r)\setminus \left\{(h,q): h^-_{r,\Scal}\leq h\leq h_r^*,\ q=\phit_r(h)\right\},
\end{equation}
 where $\Pcal^A(u^*_r)$ falls under the previous \textit{Case A}.

\section{A simple junction}\label{sec:junction}
We consider a network formed by  two canals intersecting at one single point, which represents the junction. We name the canals such that
 $1$ is the incoming canal and $2$ is the outgoing ones.
We indicate by $\su=(\sh,\sv\sh)=(\sh,\sq)$ the traces at the junction.
The flow is given by the one-dimensional shallow-water equations \eqref{eq:shallow_water} in each canal coupled with special conditions at the junction. Our aim is to define and solve Riemann problems at the junction.

A Riemann Problem at a junction is a Cauchy Problem with initial data which are constant on each canal incident at the junction.
So, assuming constant initial conditions $u^0_1$, $u^0_2$ on canal $1$ and $2$ respectively, the Riemann solution consists of intermediate states $\su_1$ and $\su_2 $  such that $\su_1 \in\Ncal(u^0_1)$ and $\su_2\in\Pcal(u^0_2)$ and verifying given coupling conditions.
Here, we are interested in evaluating possible solutions for different types of initial data belonging to different flow regimes. We start
assuming, as junction conditions, the conservation of mass
\begin{equation}\label{mass_cons}
\sq_1 = \sq_2
\end{equation}
and equal heights
\begin{equation}\label{equal_h}
\sh_1 = \sh_2.
\end{equation}
In the following we study the \textit{boundary solution} $\su$ at the junction in the following interesting cases: A$\rightarrow$A, the water flow is fluvial in both incoming and outgoing canals around the junction; B$\rightarrow$A, the water flow is torrential with positive velocity in the incoming canal while it is fluvial in the outgoing one; B$\rightarrow$B, the water flow is torrential with positive velocity in both incoming and outgoing canals. 

\subsection*{Case A$\rightarrow$A (Fluvial $\rightarrow$ Fluvial)}
Here we assume to have a left state $u_l$ and a right state $u_r$ such that $|\Fr_l|< 1$ and $|\Fr_r|< 1$.
The solution at the junction consists of two waves separated by intermediate states $\su_1$ and $\su_2$ such that 
\begin{equation}\label{1in1outAA:junc_cond_*}
\left\{\begin{array}{l}
\su_1\in\Ncal^A(u_l),
\smallskip\\
\su_2\in\Pcal^A(u_r),
\smallskip\\
\sq_1= \sq_2 = \sq,
\smallskip\\
\sh_1=\sh_2,
\end{array}\right.
\end{equation}
with $\sh_1,\sh_2>0$.
\begin{figure}
\begin{center}
\begin{overpic}
[width=0.9\textwidth]{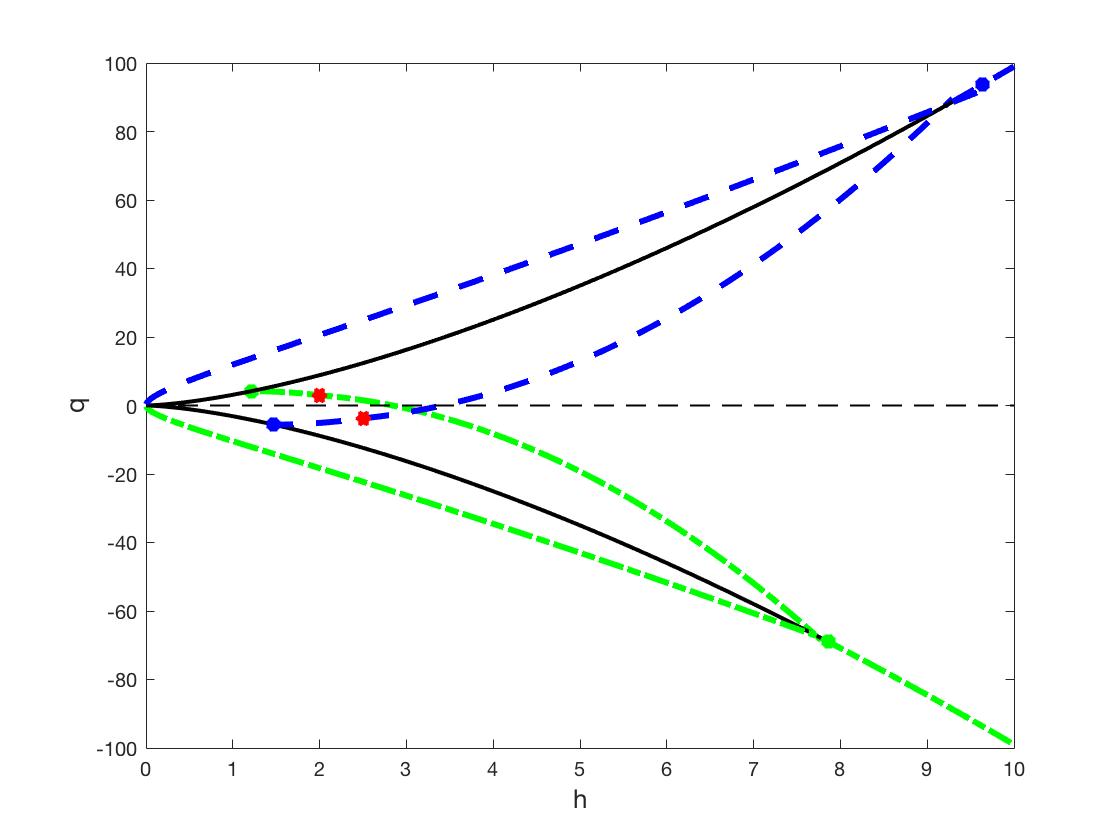}
\put(33,35){$u_r$}
\put(30,40){$u_l$}
\put(41,42){$\su$}
\put(37,40){$\bf\huge\swarrow$}
\put(20,18){$\Ncal^B(u_l)$}
\put(20,60){$\Pcal^B(u_r)$}
\end{overpic}
\end{center}
\caption{Case Fluvial $\rightarrow$ Fluvial, system \eqref{1in1outAA:junc_cond_*}. In this case curves $\phit_l$ and $\phit_r$ intersect inside the subcritical region. The solution is the intersection point $\su$.}\label{fig:AtoA_1}
\end{figure}
\begin{figure}
\begin{center}
\begin{overpic}
[width=0.9\textwidth]{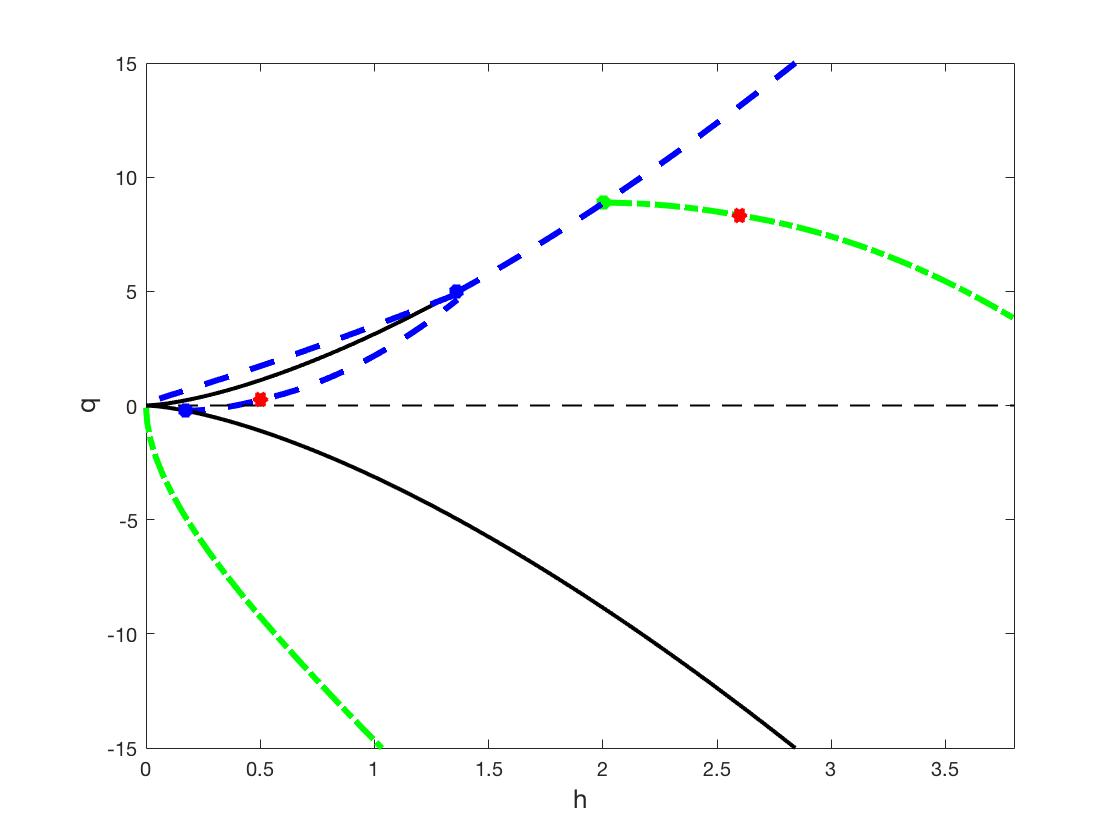}
\put(24,37){$u_r$}
\put(65,53){$u_l$}
\put(50,59){$u^+_{l,\Rcal}$}
\put(15,12){$\Ncal^A(u_l)$}
\put(20,60){$\Pcal^A(u_r)$}
\end{overpic}
\end{center}
\caption{Case Fluvial $\rightarrow$ Fluvial, system \eqref{1in1outAA:junc_cond_*}: curves $\phit_l$ and $\phit_r$ have empty intersection inside the subcritical region and $h_r<h_l$. The solution is the critical point $u^+_{l,\Rcal}$.}\label{fig:AtoA_2}
\end{figure}
\begin{figure}
\begin{center}
\begin{overpic}
[width=0.9\textwidth]{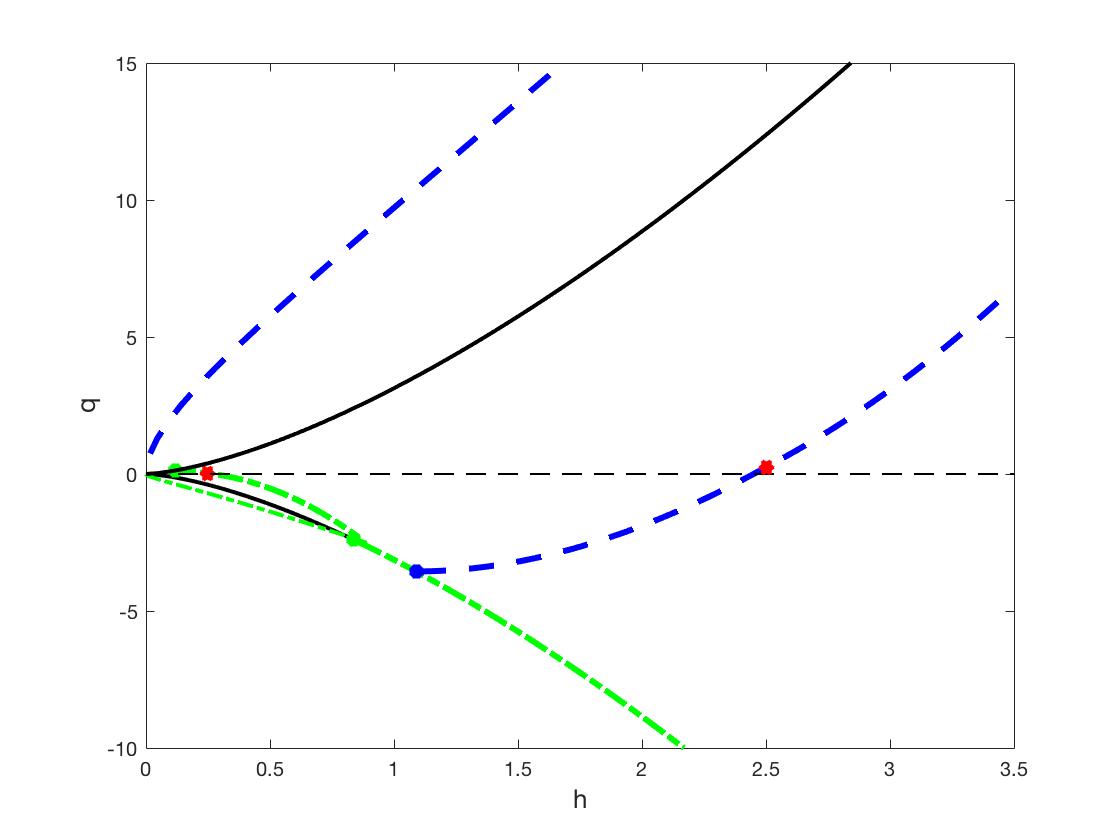}
\put(67,30){$u_r$}
\put(17,35){$u_l$}
\put(36,26){$u^-_{r,\Rcal}$}
\put(20,18){$\Ncal^A(u_l)$}
\put(20,60){$\Pcal^A(u_r)$}
\end{overpic}
\end{center}
\caption{Case Fluvial $\rightarrow$ Fluvial, system \eqref{1in1outAA:junc_cond_*}: curves $\phit_l$ and $\phit_r$ have empty intersection inside the subcritical region and $h_l<h_r$. The solution is the critical point $u^-_{r,\Rcal}$.}\label{fig:AtoA_3}
\end{figure}
\begin{prop} Under the subcritical condition on $u_l$ and $u_r$, the system \eqref{1in1outAA:junc_cond_*} admits a unique solution.
\end{prop}
\begin{proof}
We distinguish two cases:\\
{\bf Case 1.} The two curves $\phit_l$ and $\phit_r$ intersect inside the subcritical region (Figure \ref{fig:AtoA_1}). In this case the solution is trivially the intersection point.\\
{\bf Case 2.} The two curves $\phit_l$ and $\phit_r$ do not intersect inside the subcritical region.
If $h_r<h_l$ (specifically $h^+_{r,\Scal}<h^+_{l,\Rcal}$, see Figure \ref{fig:AtoA_2}) then the only point that verifies the junction conditions \eqref{1in1outAA:junc_cond_*} is the \textit{critical} 
point $u^+_{l,\Rcal}$ given in \eqref{uplr}, thus
\begin{equation}
\su_1=\su_2=u^+_{l,\Rcal}.
\end{equation}
If $h_r>h_l$ (specifically $h^-_{r,\Rcal}>h^-_{l,\Scal}$, see Figure \ref{fig:AtoA_3}) the only point that verifies the junction conditions \eqref{1in1outAA:junc_cond_*} is the \textit{critical} point $u^-_{r,\Rcal}$ given in \eqref{umrr}, then 
\begin{equation}
\su_1=\su_2=u^-_{r,\Rcal}.
\end{equation}
\end{proof}

\begin{remark} Notice that the proposed procedure may give a solution which is different from the classical solution of the Riemann problem on a single channel, given by the intersection point of $\phi_l$ and $\phi_r$ curves, see Remark \ref{remark:std_sol}.
\end{remark}


\subsection*{Case B$\rightarrow$A (Torrential $\rightarrow$ Fluvial)}
Here we assume to have a left state $u_l$ and a right state $u_r$ with $\Fr_l>1$ and $|\Fr_r|<1$.
The solution at the junction consists of two waves separated by intermediate states $\su_1$ and $\su_2$ such that 
\begin{equation}\label{1in1outBA:junc_cond_*}
\left\{\begin{array}{l}
\su_1\in\Ncal^B(u_l),
\smallskip\\
\su_2\in\Pcal^A(u_r),
\smallskip\\
\sq_1= \sq_2 = \sq,
\smallskip\\
\sh_1 = \sh_2,
\end{array}\right.
\end{equation}
with $\sh_1,\sh_2>0$. Next Proposition provides the results about
solutions, while the illustrating figures are postponed to the Appendix.
\begin{prop} System \eqref{1in1outBA:junc_cond_*} 
admits a solution if the two regions $\Ncal^B$ and $\Pcal^A$ intersect in the subcritical set $\{(h,q): h>0, \ \Ct^-(h) \leq q \leq \Ct^+(h)\}$ or if $u_l\in\Pcal^A{(u_r)}$.
\end{prop}
\begin{proof}
We distinguish two cases:\\
{\bf Case 1.} The two curves $h\Scal_1(h_l^*,v_l^*;h)$ and $\phit_r(h)$ intersect inside the subcritical region (Figure \ref{fig:BtoA_1} of the Appendix). 
If $u_l\notin\Pcal^A(u_r)$ the solution is trivially the intersection point. On the contrary, if $u_l\in\Pcal^A(u_r)$ we have two possible solutions: the intersection point inside the subcritical region or the starting supercritical value $u_l$. So, we may obtain both fluvial or torrential regime. Coherently with the solution that would be obtained in a single channel we choose the fluvial regime assigning as solution the intersection point inside the subcritical region.\\
{\bf Case 2.} The two curves $h\Scal_1(h_l^*,v_l^*;h)$ and $\phit_r(h)$ do not intersect inside the subcritical region. We distinguish the two subcases shown in Figure \ref{fig:BtoA_2} and in Figure \ref{fig:BtoA_3} of the Appendix.\\
{\bf Case 2.1.} Referring to Figure \ref{fig:BtoA_2}, the point $u_l$ may fall or not in the $\Pcal^B(u_r)$ region. If $u_l\in\Pcal^A(u_r)$ the only possible solution is $u_l$ self and $\su_1=\su_2= u_l$; if on the contrary $u_l\notin\Pcal^A(u_r)$, system  \eqref{1in1outBA:junc_cond_*} does not admit a solution.\\
{\bf Case 2.2.} Referring Figure \ref{fig:BtoA_3}, if $u_l\notin\Pcal^A(u_r)$ the only possible solution is the \textit{critical} point $u^-_{r,\Rcal}$ defined in \eqref{umrr}, then $\su_1=\su_2=u^-_{r,\Rcal}$;
If $u_l\in\Pcal^A(u_r)$, both $u_l$ self and the critical point $u^-_{r,\Rcal}$ are admissible solutions. 
Coherently with the solution obtained in the previous subcase, we select again $u_l$ as solution and we assign $\su_1=\su_2= u_l$. So, the torrential regime propagates on the outgoing canal.

\end{proof}

\subsection*{Case B$\rightarrow$B (Torrential $\rightarrow$ Torrential)}
As shown in Figure \ref{fig:BtoB} of the Appendix, the two admissible regions $\Ncal^B$ and $\Pcal^B$ may have empty intersection. So, in this case there exists 
a solution at the junction which verifies the conditions of conservation of the mass and of equal heights, if and only if $u_l\in\Pcal^B(u_r)$, i.e. $\su_1=\su_2=u_l$.


\subsection{Other possible conditions at the junction}
Assuming different conditions at the junction give rise to new possible solutions. In canals network problems, it is usual to 
couple the conservation of the mass with the conservation of energy at the junctions.
The \textit{specific energy} $E$ is a useful parameter in channel flow and it is defined as
\begin{equation}\label{eq:energy}
E = h +\frac{v^2}{2g}.
\end{equation}
For a given flow rate, there are usually two states possible for the same specific energy. 
Studying $E$ as a function of $h$ for constant $q$, there is a minimum value of $E$ at a certain value of $h$ called the \textit{critical depth},
\begin{equation}\label{eq:critical_h}
h=h_c =\left(\frac{q^2}{g}\right)^{\frac 1 3}.
\end{equation}
￼Critical depth $h_c$ corresponds to some critical channel velocity $v_c$ defined by $\Fr_c=1$.
For $E < E_{min}$ solution does not exists, and thus such a flow is impossible physically. For $E > E_{min}$ solutions are possible: large depth with $|\Fr|< 1$ subcritical flow, and small depth with $|\Fr| > 1$ supercritical flow.

In our case,  assuming equal energy at the junction gives
\begin{equation}\label{eq:energy_cons}
\frac{v_1^2}{2}+g h_1=\frac{v_2^2}{2}+g h_2.
\end{equation}
Moreover, assuming $\sq_1=\sq_2=\sq$ to be constant (and known)
we get from \eqref{eq:energy_cons}
$$
\frac{g h_1\Fr_1^2}{2}+g h_1=\frac{g h_1^3\Fr_1^2}{2 h_2^2}+g h_2,
$$
where $\Fr_1=v_1/\sqrt{g h_1}$ and where we used the following relations
$$ v_1^2 =g h_1\Fr_1^2 \quad\mbox{ and }\quad v_2^2=\frac{v_1^2 h_1^2}{h_2^2} = \frac{g h_1^3\Fr_1^2}{ h_2^2}.$$
Then, we have two possible solution for the heights values at the junction:
\begin{equation}\label{eq:equal_h}
\sh_1=\sh_2 \ \mbox{ (equal heigths)}
\end{equation}
or
\begin{equation}
\frac{\sh_2}{\sh_1} = \frac{\Fr_1^2}{4}\left(1+\sqrt{1+\frac{8}{\Fr_1^2}}\right).
\end{equation}
So, for $\sh_1\ne\sh_2$ we get new possible solutions at the junction with $(\sh_1,\sq)$ subcritical and $(\sh_2,\sq)$ supercritical or vice-versa.
Specifically, case Torrential$\rightarrow$Fluvial and case Torrential$\rightarrow$Torrential may admit solution even if their admissible regions have empty intersection in the subcritical set.

\begin{remark}
In the case of a simple junction, the natural assumption (consistent with the dynamic of shallow-water equations) should be to assume the conservation of the momentum. With our notation, the relation \eqref{eq:equal_h} or \eqref{eq:energy_cons} sholud be replaced by the following:
\begin{equation}\label{eq:momentun_cons}
\frac{q_1^2}{h_1}+\frac 1 2 g h_1^2 = \frac{q_2^2}{h_2}+\frac 1 2 g h_1^2.
\end{equation}
By the same reasoning used before in the case of the conservation of energy, from \eqref{eq:momentun_cons} we get
$$
\left(\frac{h_2}{h_1}\right)^3-\left(2\Fr_1^2+1\right)\left(\frac{h_2}{h_1}\right)+2\Fr_1^2 = 0.
$$
Then, we have again two possible relations for the heights values at the junction:
\begin{equation}
\sh_1=\sh_2 \ \mbox{ (equal heigths)}
\end{equation}
or
\begin{equation}\label{eq:ratio_momentum}
\frac{\sh_2}{\sh_1} = \frac 12 \left(-1+\sqrt{1+8\Fr_1^2}\right).
\end{equation}
So again, for $\sh_1\ne\sh_2$ we get new possible solutions at the junction. Case Torrential$\rightarrow$ Fluvial and case Torrential$\rightarrow$Torrential may have solution and specifically we have that in Figure \ref{fig:BtoA_2}, points $(h^*_l,q_l)$ and $(h_l,q_l)$ verify \eqref{eq:ratio_momentum} (see \eqref{eq:hstarl}) and if $u_l\in\Pcal^A(u_r)$ they are candidate to be the new solution. Even in the case described in Figure \ref{fig:BtoB}, if $u_l\in\Pcal^B(u_r)$ the points $(h^*_l,q_l)$ and $(h_l,q_l)$ are the solution at the junction.
\\
Let us conclude observing that for  appropriate values of $(h_l,q_l)$,  for Torrential$\rightarrow$ Fluvial we would get the same solution considering our simple network as a simple canal, i.e. a stationary shock called \textit{hydraulic jump} characterized indeed by the conservation of the momentum in the transition from a supercritical to subcritical flow \cite{Dasgupta2011}.
\end{remark}

\section{Numerical tests}\label{sec:num_tests}
In this Section we illustrate the results of Section \ref{sec:junction}
by means of numerical simulations. We first give a sketch of the adopted numerical procedure and then we focus on two numerical tests 
which illustrate the regime transitions from fluvial to torrential and viceversa.
The latter depend on well chosen initial conditions for Riemann problems at the junction.

We consider again a network formed by  two canals intersecting at one single point, which represents the junction.
Following \cite{BPQ2016}, we use a high order Runge-Kutta Discontinuous Galerkin scheme to numerically solve system \eqref{eq:shallow_water_homog} on both canals $1$ and $2$:
\begin{equation}\label{eq:test_sys}
\begin{array}{cc}
\partial_t u_1 +\partial_x f(u_1)=0, & \mbox{for } x<0,
\smallskip\\
\partial_t u_2 +\partial_x f(u_2)=0, &  \mbox{for } x<0.
\end{array}
\end{equation}
The 1D domain of each canal is discretized into cells $C_m=[x_{m-\frac12},x_{m+\frac12}]$, with $x_m=(x_{m-\frac12}+x_{m+\frac12})/2$, $m=1,\ldots,M$, being $M$ the total number of computational cells. The DG method for \eqref{eq:shallow_water_homog} is formulated by multiplying the equation system by some test functions $w$, integrating over each computational cell, and performing integration by parts. Specifically, we seek the approximation $U= (u_1, u_2)$ with $u_i\in{W}_{\Dx} = \{w: w|_{C_m}\in P^k(C_m), m=1,\ldots,M\}$, $i=1,2$, where $P^k(C_m)$ is the space of polynomials of degree at most $k$ on cell $C_m$, such that $\forall\ w\in W_{\Dx}$ 
\begin{eqnarray}
\label{DG_std}
\int_{C_m}w(x)\partial_t U dx = \int_{C_m} f(U)\partial_x w(x) dx &-& \left(\hat f_{m+\frac12} w^-_{m+\frac12}-\hat f_{m-\frac12} w^+_{m-\frac12} \right).
\end{eqnarray}
Terms $w^{\pm}_{m+\frac12}$ denote left and right limits of the function values and the numerical fluxes $\hat f_{m\pm\frac12} = f(U^-_{m\pm\frac12},U^+_{m\pm\frac12})$ are approximate Riemann solvers. In our simulations, we use the Lax-Friedrich flux.
For implementation, each component of the approximate solution $U$, e.g. $u_1$, on mesh $C_m$ is expressed as
\begin{equation}
u_1(x,t)=\sum_{l=0}^{k} \hat{u}_{m}^{1, l}(t)\psi_m^l(x),
\end{equation}
where $\{\psi_m^l(x)\}_{l=0}^k$ is the set of basis functions of $P^k(C_m)$. Specifically, we choose the Legendre polynomials as local orthogonal basis of $P^k(C_m)$ and take the test function $w(x)$ in eq.~\eqref{DG_std} exactly as the set of basis functions $\psi_m^l(x)$, $l=0 \cdots k$, assuming the polynomial degree $k=2$. The equation system \eqref{DG_std} can then be evolved in time via the method of lines approach by a TVD RK method. More implementation details for RKDG methods can be found in the original paper \cite{CS1989} and the review article \cite{CS2001}.

Once the numerical procedure on both canals has been settled, the two systems in \eqref{DG_std} have to be coupled with boundary conditions. At the junction the boundary values is settled as follows: 
at each time step and at each RK stage via the method-of-line approach, 
we set as left state in \eqref{1in1outAA:junc_cond_*} (or \eqref{1in1outBA:junc_cond_*}) the approximate solution from canal $1$ at the left limit of the junction, i.e. 
$$
u_l \approx U_l= \lim_{x\rightarrow x_{M+\frac12}^{-}} U_1(x,\cdot)
$$ 
with $[x_{M-\frac12}, x_{M+\frac12}]$ being the right-most cell in the 1D discretization of the incoming canal $1$, and as right state in \eqref{1in1outAA:junc_cond_*} (or \eqref{1in1outBA:junc_cond_*}) the 
approximate solution from canal $2$ at the right limit of the junction, i.e. 
$$
u_r \approx U_r=\lim_{x\rightarrow x_{\frac12}^{+, 2}} U_2(x,\cdot), \quad
$$ 
with $[x_{\frac12}, x_{\frac32}]$ being the left-most cell in the 1D discretization of the outgoing canal $2$. With this two values, we compute the intermediate states $u^b_1$ and $u^b_2$ by solving \eqref{1in1outAA:junc_cond_*} (or \eqref{1in1outBA:junc_cond_*}) and, preserving the mass we
directly assign the numerical fluxes at the junction as 
$$
\hat{f}_{M_+\frac12} \doteq f(u_1^b) \quad \mbox{for the canal $1$},
\quad
\hat{f}_{\frac12} \doteq f(u_2^b) \quad \mbox{for the canal $2$}.
$$
Finally, in our simulations we assume Neumann boundary conditions at the free extremity of the channels.

Applying this numerical procedure, in Figure \ref{fig:testAinA} and \ref{fig:testBinA} we give two examples which illustrate the solution that is obtained in the regime transitions from fluvial to torrential and viceversa.
In Figure \ref{fig:testAinA}, we assume to have a starting configuration given by the following subcritical constant states: $u_l=(0.25,0.025)$ on canal $1$ (left) and $u_r=(2.5,0.25)$ on canal $2$ (right).
We are in the situation showed in Figure \ref{fig:AtoA_3} and the boundary value at the junction is defined by the critical value $u^-_{r\Rcal}\approx(1.088,-3.55)$. The solution is a backward water movement along canal $1$ with a torrential regime.
In Figure \ref{fig:testBinA}, we assume as initial state on canal $1$ the supercritical constant value $u_l=(0.2,3)$ and on canal $2$ the subcritical constant value $u_r=(1.8,4)$. We are in the configuration given in Figure \ref{fig:BtoA_2} and the boundary value at the junction is defined by the starting point $u_l$. The solution is a forward movement with positive velocity, so the torrential regime propagates on canal $2$.    
\begin{figure}
\begin{center}
\begin{overpic}
[width=0.9\textwidth]{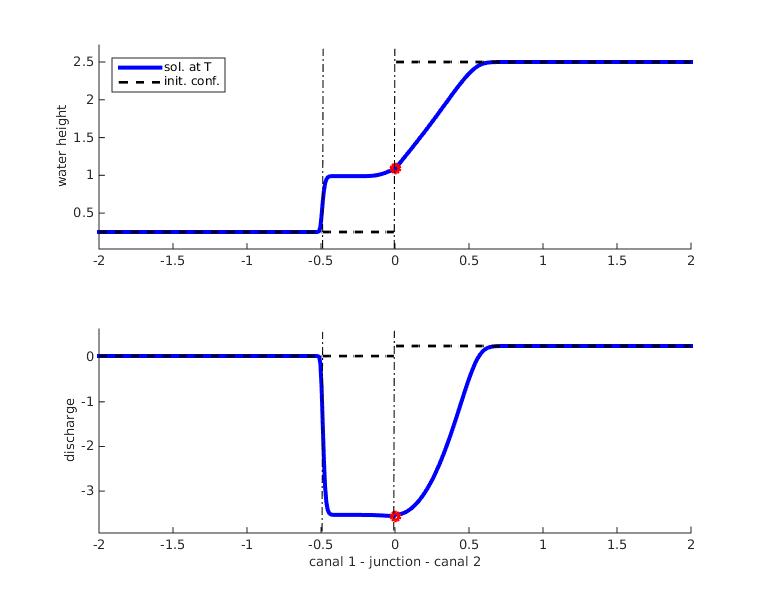}
\put(25,60){$F_r<1$}
\put(43,60){$F_r>1$}
\put(65,60){$F_r<1$}
\put(38,51){$\leftarrow$}
\put(25,20){$F_r<1$}
\put(43,20){$F_r>1$}
\put(65,20){$F_r<1$}
\end{overpic}
\end{center}
\caption{Numerical test case for the configuration given in Fig. \ref{fig:AtoA_3}.}\label{fig:testAinA}
\end{figure}
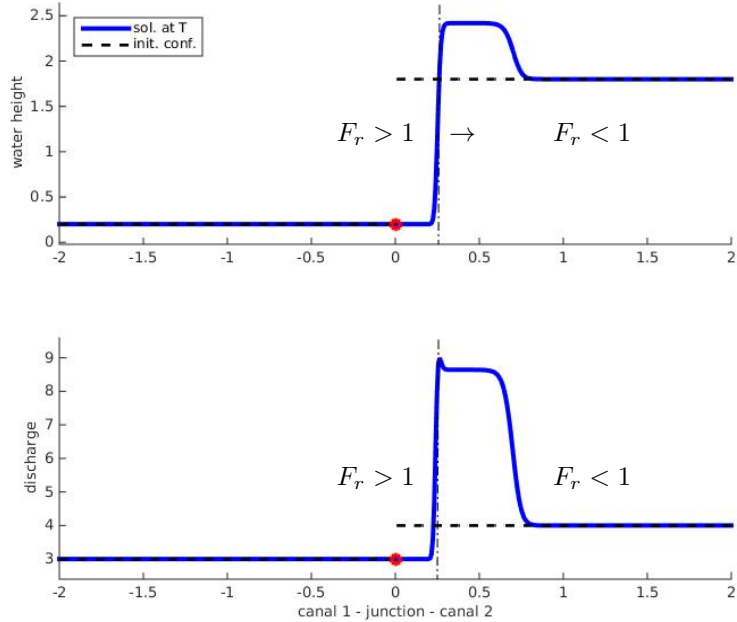
\begin{figure}
\begin{center}
\begin{overpic}
[width=0.9\textwidth]{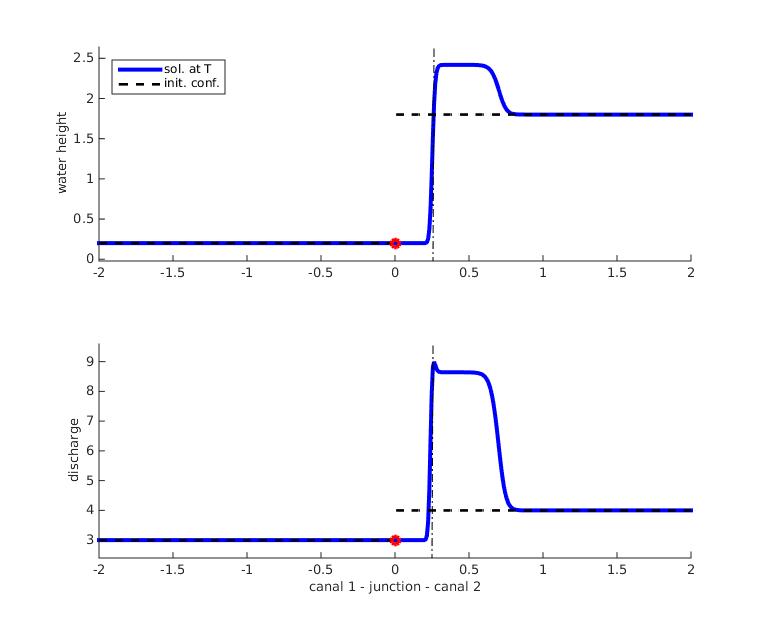}
\put(45,60){$F_r>1$}
\put(70,60){$F_r<1$}
\put(58,60){$\rightarrow$}
\put(45,20){$F_r>1$}
\put(70,20){$F_r<1$}
\end{overpic}
\end{center}
\caption{Numerical test case for the configuration given in Fig. \ref{fig:BtoA_2}.}\label{fig:testBinA}
\end{figure}


\section{Conclusions}
This paper deals with open canal networks. The interest stems out
of applications such as irrigation channels water management.
We base our investigations on the well-known Saint-Venant or
shallow water equations. Two regimes exist for this hyperbolic
system of balance laws: the fluvial, corresponding to eigenvalues
with different sign, and the torrential, corresponding to both positive
eigenvalues.
Most authors focused the attention on designing and analysing
network dynamics for the fluvial regime, while here we extend the theory
to include regime transitions.
After analyzing the Lax curves for incoming and outgoing canals,
we provide admissibility conditions for Riemann solvers, describing
solutions for constant initial data on each canal.
Such analysis allows to define uniquely dynamics according to a set of conditions
at junctions, such as conservation of mass, equal water height
or equal energy. More precisely, the simple case
of one incoming and outgoing canal is treated showing that, 
already in this simple example,
regimes transitions appear naturally at junctions.
Our analysis is then visualized by numerical simulations based on
Runge-Kutta Discontinuous Galerkin methods.


\bibliography{half-riemann}{}
\bibliographystyle{plain}

\section*{Appendix}
Here we collect additional figures illustrating attainable regions for half-riemann problems and solutions for a simple channel. 
Figures \ref{fig:out_A}--\ref{fig:out_C} refer to the right-half Riemann problem described in Section \ref{sec:rightRP}. 
They show the regions of admissible states such that waves on the outgoing canals do not propagate into the junction, given a right state $u_r$ such that $|\tilde\Fr_r|<1$, $\tilde\Fr_r>1$ and $\tilde\Fr_r<-1$ respectively.
\\
Figures \ref{fig:BtoA_1}--\ref{fig:BtoB} refer to Section \ref{sec:junction} in which we study the possible solutions at a simple junction for different flow regimes, assuming the conservation of mass and equal heights at the junction. Specifically,  Figures \ref{fig:BtoA_1}--\ref{fig:BtoA_3} illustrate the  possible configurations and their associated solution that may occur during the transition from torrential to fluvial regime. The last Figure \ref{fig:BtoB} shows instead the only possible configuration that admits a solution for the torrential flow regime.
 
\begin{figure}
\begin{center}
\begin{overpic}
[width=0.9\textwidth]{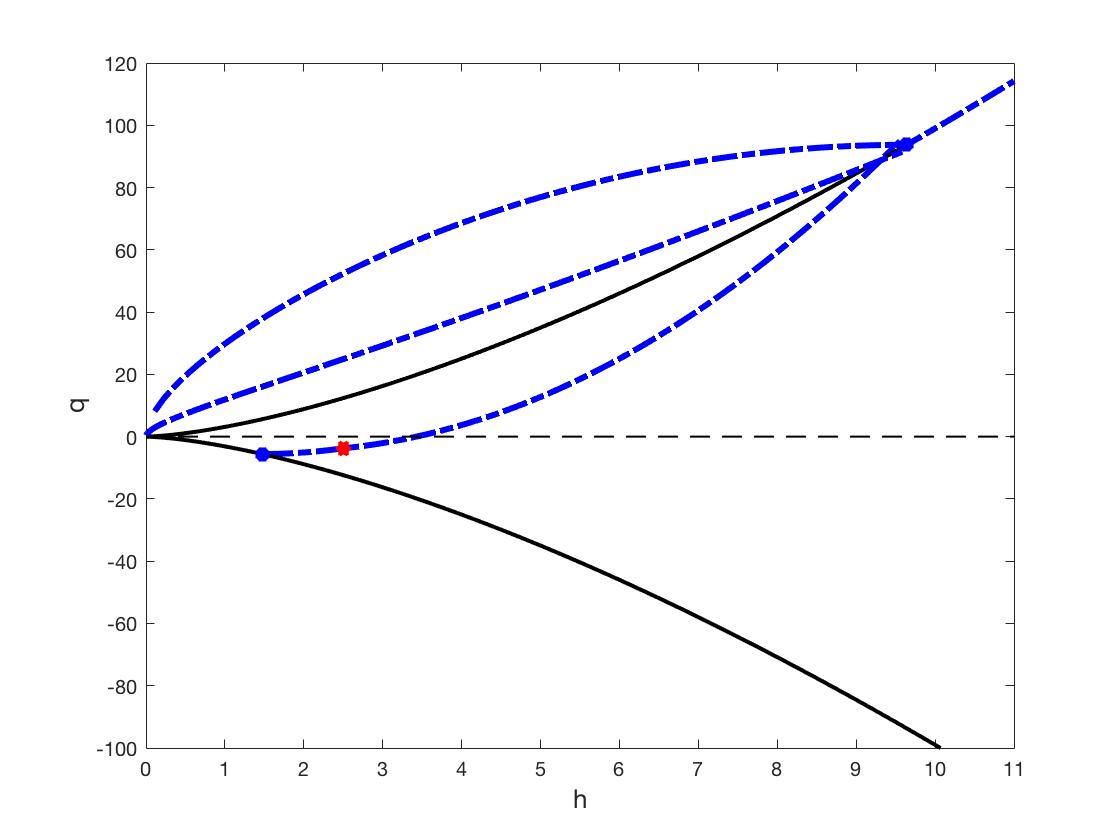}
\put(32,33){$u_r$}
\put(22,31){$u^-_{r,\Rcal}$}
\put(81,59){$u^+_{r,\Scal}$}
\put(25,60){$\Ocal^A_2$}
\put(35,48){$\Ocal^A_3$}
\put(65,40){$\Ocal^A_1$}\put(62,43){$\huge\nwarrow$}
\put(83,65){$\Ct^+$}
\put(49,60){\makebox(0,0){\rotatebox{15}{$q=\tilde\Scal^{-1}_1(u^{+}_{r,\Scal};h)$}}}
\end{overpic}
\end{center}
\caption{Right-half Riemann problem, Section \ref{sec:rightRP}. Region $\Pcal^A(u_r)=\Ocal^A_1\bigcup\Ocal^A_2\bigcup\Ocal^A_3$ defined by \eqref{eq:OA1}-\eqref{eq:OA3} where $u_r$ is such that $|\tilde\Fr_r|<1$.}\label{fig:out_A}
\end{figure}
\begin{figure}
\begin{center}
\begin{overpic}
[width=0.9\textwidth]{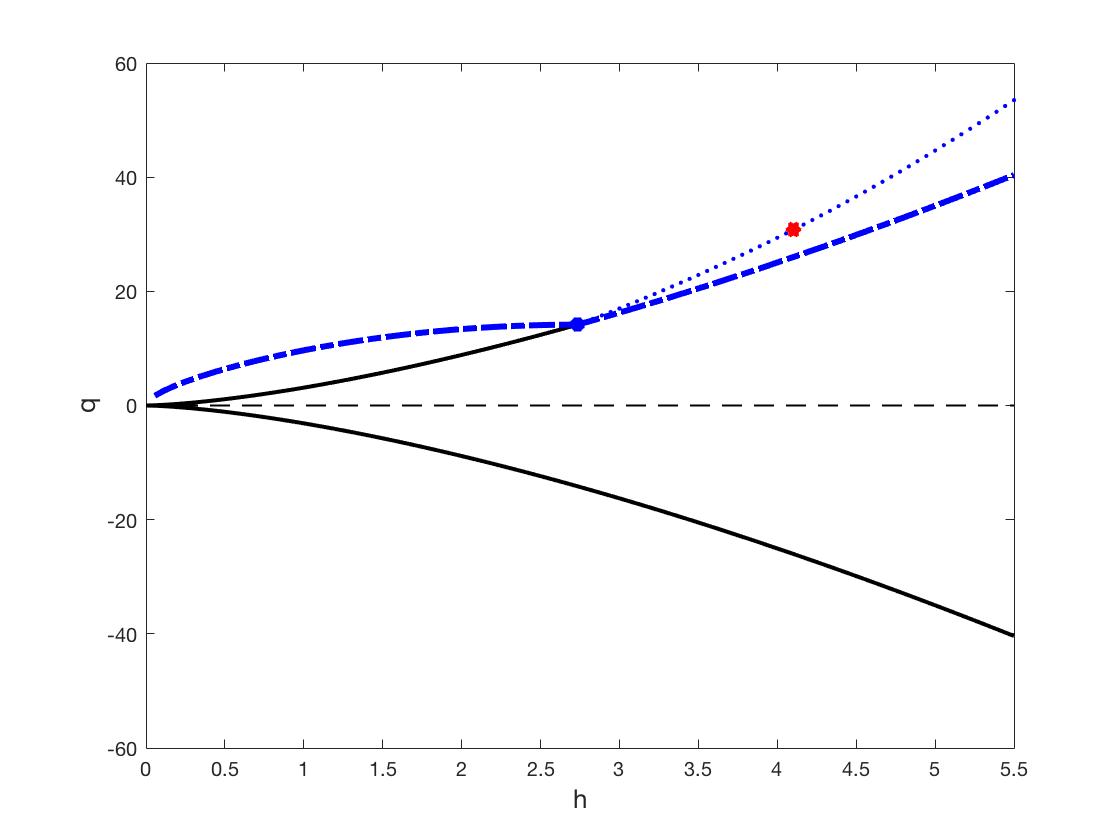}
\put(69,57){$u_r$}
\put(51,43){$u^+_{r,\Rcal}$}
\put(38,58){$\Pcal^B(u_r)$}
\put(78,51){$\Ct^+$}
\put(30,47){\makebox(0,0){\rotatebox{10}{$q=\tilde\Scal^{-1}_1(u^{-}_{r,\Rcal};h)$}}}
\end{overpic}
\end{center}
\caption{Right-half Riemann problem, Section \ref{sec:rightRP}. Region $\Pcal^B(u_r)$ bounded by $q=\tilde\Scal_2(u^{-}_{l,\Rcal};h)$ and $q=\Ct^+(h)$ as defined in \eqref{eq:PB} where $u_r$ is such that $\tilde\Fr_r>1$.}
\label{fig:out_B}
\end{figure}
\begin{figure}
\begin{center}
\begin{overpic}
[width=0.9\textwidth]{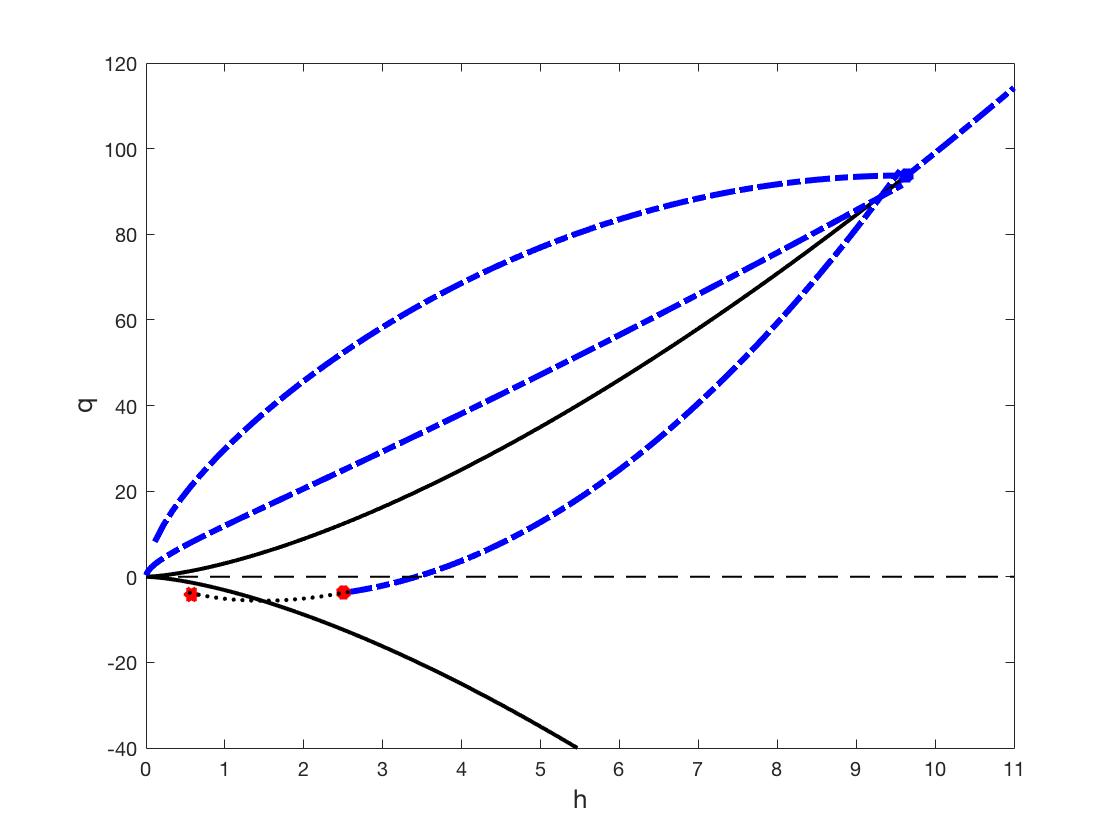}
\put(16,19){$u_r$}
\put(31,19){$u^*_r$}
\put(80,56){$u^{*,+}_{r,\Scal}$}
\put(25,60){$\Ocal^{*,A}_2$}
\put(45,45){$\Ocal^{*,A}_3$}
\put(64,29){$\Ocal^{*,A}_1$}\put(61,32){$\huge\nwarrow$}
\put(83,63){$\Ct^+$}
\put(43,53){\makebox(0,0){\rotatebox{24}{$q=\tilde\Scal^{-1}_1(u^{*,+}_{r,\Scal};h)$}}}
\end{overpic}
\end{center}
\caption{Right-half Riemann problem, Section \ref{sec:rightRP}. Region $\Pcal^C(u_r)=\Ocal^{*,A}_1\bigcup\Ocal^{*,A}_2\bigcup\Ocal^{*,A}_3$ given in \eqref{eq:PC} where $u_r$ is such that $\tilde\Fr_r<-1$.}
\label{fig:out_C}
\end{figure}
\begin{figure}
\begin{center}
\begin{overpic}
[width=0.9\textwidth]{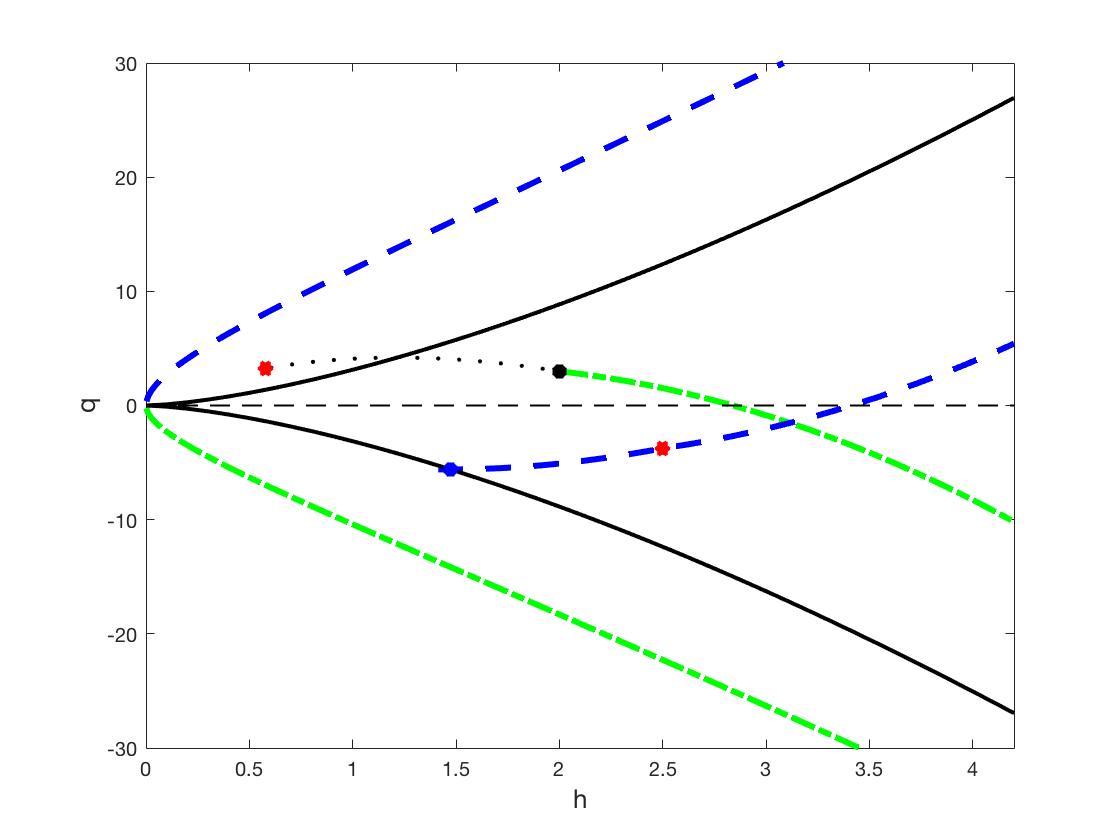}
\put(58,32){$u_r$}
\put(23,43){$u_l$}
\put(50,43){$u^*_l$}
\put(70,34){$\bf\huge\uparrow$}
\put(70,31){$\su$}
\put(20,18){$\Ncal^B(u_l)$}
\put(30,60){$\Pcal^A(u_r)$}
\end{overpic}
\end{center}
\caption{Case Torrential $\rightarrow$ Fluvial, system \eqref{1in1outBA:junc_cond_*}. In this case the curve $h\Scal_1(h^*,v^*;h)$ and $\phit_r(h)$ intersect inside the subcritical region. The solution is the intersection point $\su$.}\label{fig:BtoA_1}
\end{figure}
\begin{figure}
\begin{center}
\begin{overpic}
[width=0.9\textwidth]{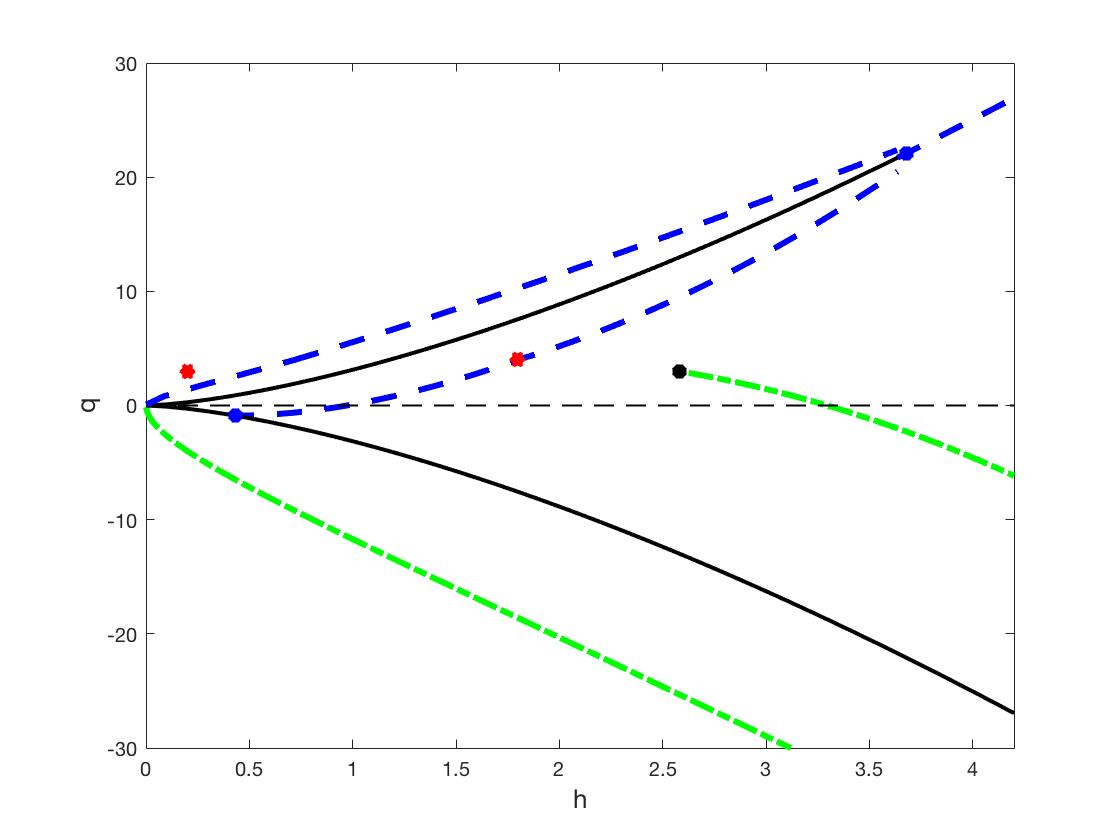}
\put(47,41){$u_r$}
\put(15,43){$u_l$}
\put(60,43){$u^*_l$}
\put(20,18){$\Ncal^B(u_l)$}
\put(40,54){$\Pcal^A(u_r)$}
\end{overpic}
\end{center}
\caption{Case Torrential $\rightarrow$ Fluvial, system \eqref{1in1outBA:junc_cond_*}. In this case the curve $h\Scal_1(h^*,v^*;h)$ and $\phit_r(h)$ have empty intersection inside the subcritical region. This configuration is an example in which system \eqref{1in1outBA:junc_cond_*} does not admit a solution.}\label{fig:BtoA_2}
\end{figure}
\begin{figure}
\begin{center}
\begin{overpic}
[width=0.9\textwidth]{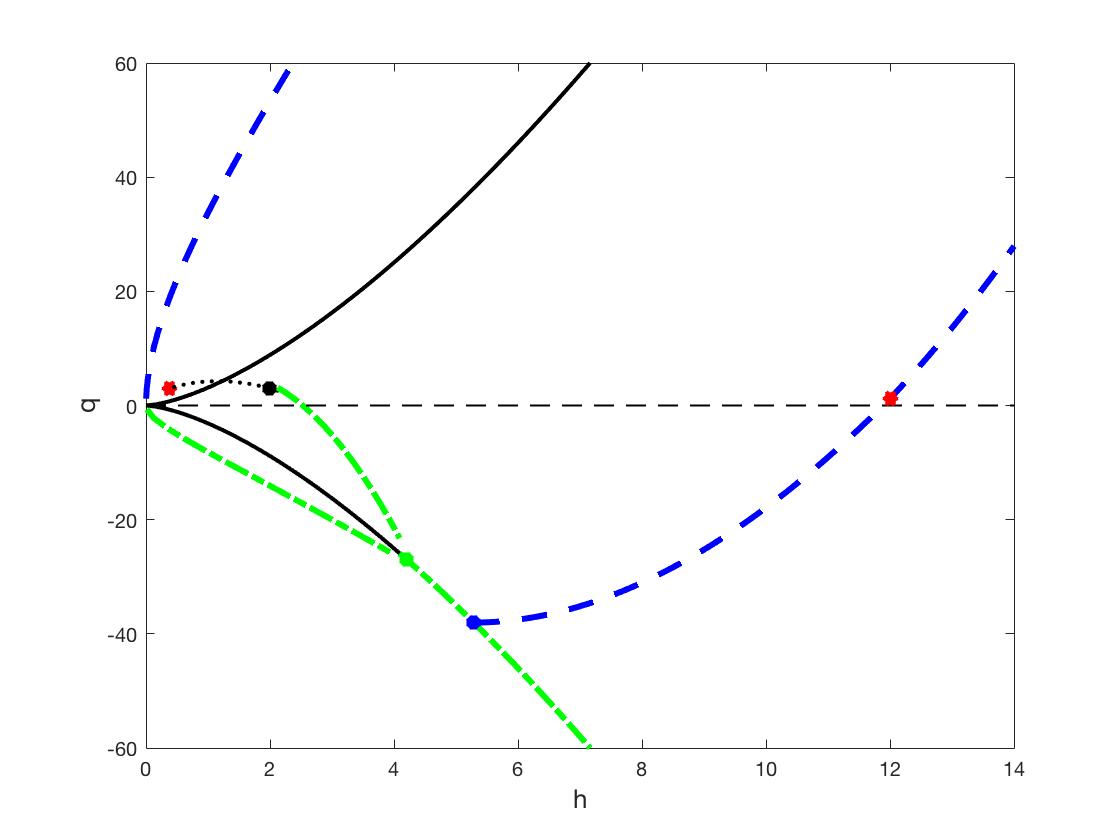}
\put(77,41){$u_r$}
\put(15,42){$u_l$}
\put(24,41){$u^*_l$}
\put(36,18){$u^-_{r,\Rcal}$}
\put(20,18){$\Ncal^B(u_l)$}
\put(18,65){\makebox(0,0){\rotatebox{45}{$\Pcal^B(u_r)$}}}
\end{overpic}
\end{center}
\caption{Case Torrential $\rightarrow$ Fluvial, system \eqref{1in1outBA:junc_cond_*}. In this case the curve $h\Scal_1(h^*,v^*;h)$ and $\phit_r(h)$ have empty intersection inside the subcritical region. The solution is the point $u^-_{r,\Rcal}$. }\label{fig:BtoA_3}
\end{figure}
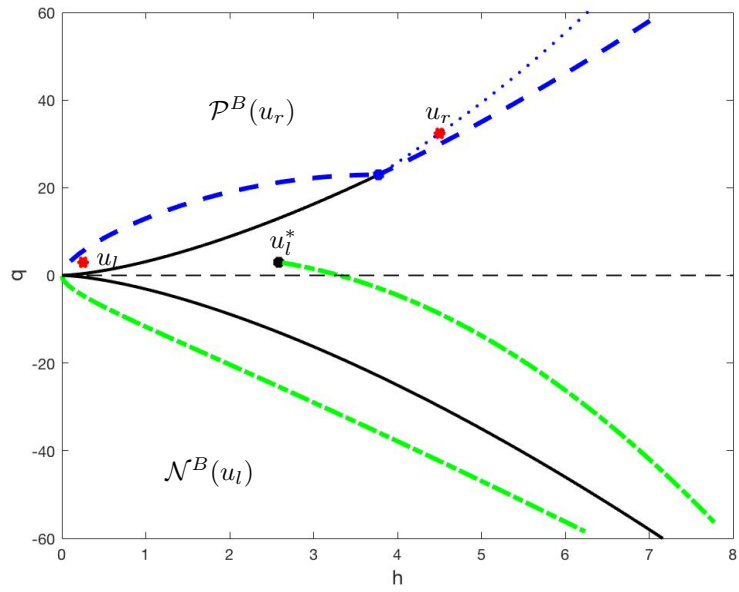
\begin{figure}
\begin{center}
\begin{overpic}
[width=0.9\textwidth]{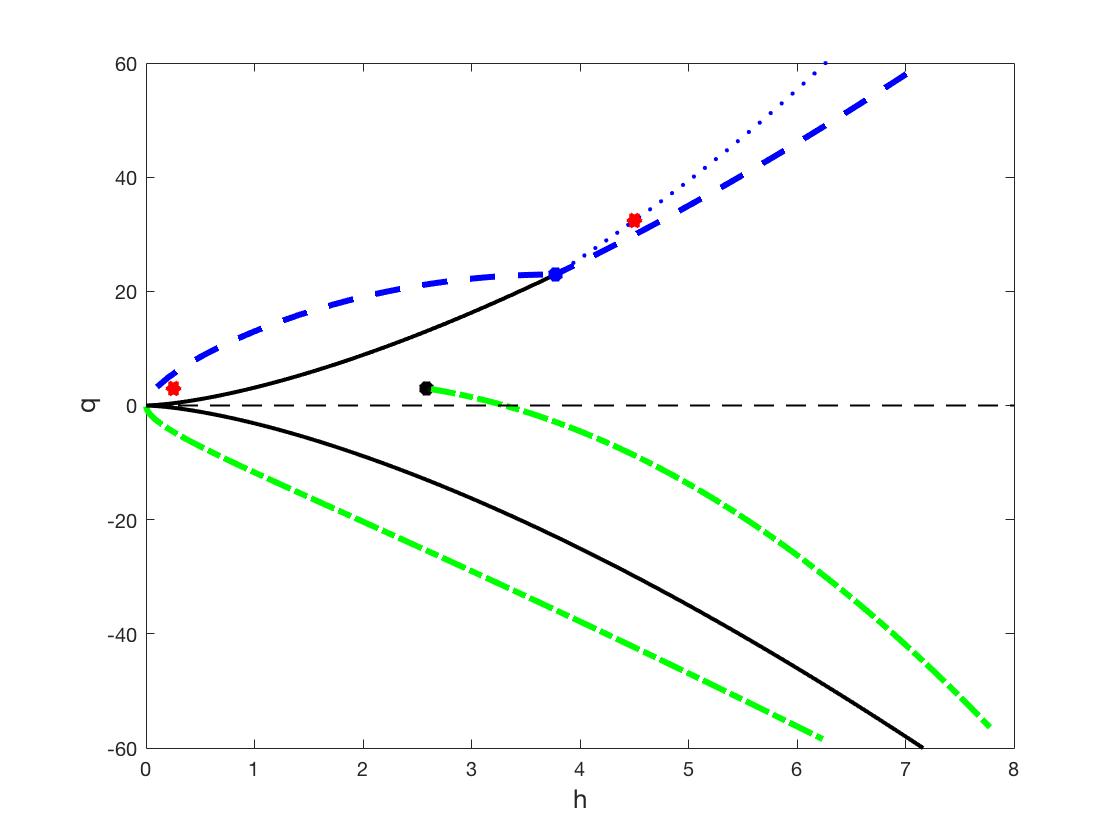}
\put(55,57){$u_r$}
\put(17,40){$u_l$}
\put(37,42){$u^*_l$}
\put(25,15){$\Ncal^B(u_l)$}
\put(30,57){$\Pcal^B(u_r)$}
\end{overpic}
\end{center}
\caption{Case Torrential $\rightarrow$ Torrential. In this case the two admissible regions $\Ncal^B$ and $\Pcal^B$ have empty intersection.}\label{fig:BtoB}
\end{figure}

\end{document}